\documentclass[10pt]{article}

\usepackage[hidelinks]{hyperref}
\usepackage{amsmath}
\usepackage{eqnarray,amsmath}
\usepackage{amssymb}
\usepackage{amsfonts}
\usepackage[usenames,dvipsnames]{pstricks}
\usepackage{graphicx}
\usepackage{subcaption}
\usepackage[labelformat=parens,labelsep=quad,skip=3pt]{caption}
\usepackage{qtree}
\usepackage{algorithm}
\usepackage[noend]{algpseudocode}
\usepackage[utf8]{inputenc}
\usepackage{amsthm}
\usepackage{pstricks-add,}
\usepackage{mathtools}
\usepackage{dirtytalk}

\usepackage{xcolor}
\usepackage{hyperref}

\usepackage[english]{babel} 
\usepackage[pangram]{blindtext}

\usepackage{array}
\usepackage{rotating}
\usepackage{tikz}
\usepackage{blkarray}

\usepackage{MnSymbol} 

\usepackage{bbm}

\usepackage{nicematrix}
\NiceMatrixOptions{
code-for-first-row = \color{blue} ,
code-for-last-row = \color{blue} ,
code-for-first-col = \color{blue} ,
code-for-last-col = \color{blue}
}

\hypersetup{colorlinks=false,linkbordercolor=red,linkcolor=green,pdfborderstyle={/S/U/W 1}}

\usepackage{atbegshi}
\AtBeginDocument{\AtBeginShipoutNext{\AtBeginShipoutDiscard}}

\def\eref#1{(\ref{#1})}

\def\a{{\alpha}}

\newtheorem{theorem}{Theorem}[section]

\newtheorem{lemma}[theorem]{Lemma}
\newtheorem{remark}[theorem]{Remark}
\newtheorem{corollary}[theorem]{Corollary}
\newtheorem{example}[theorem]{Example}
\newtheorem{definition}[theorem]{Definition}

\def\eref#1{(\ref{#1})}

\newcommand{\br}{\mathbf{r}}
\newcommand{\bc}{\mathbf{c}}

\newcommand{\Op}{\mathcal{O}}

\newcommand{\id}{{\rm id}}

\newcommand{\pq}{\preceq}

\tikzset{use/.style={}}
\tikzset{rectangle/.append style={circle, fill=blue!25}}

\newcommand{\PosetInsertion}[3]{
    \ensuremath{
    \mathop{
    \begin{tikzpicture}[scale=0.25]
        \ifnum#1=1 \fill[black!75] (0,0.5) rectangle (0.5,1); \fi
        \ifnum#2=1 \fill[black!75] (0,0) rectangle (0.5,0.5); \fi
        \ifnum#3=1 \fill[black!75] (0.5,0) rectangle (1,0.5); \fi
        \draw[thin](0,0.5) rectangle (0.5,1);
        \draw[thin](0,0) rectangle (0.5,0.5);
        \draw[thin](0.5,0) rectangle (1,0.5);
    \end{tikzpicture}
    }}
}

\newcommand{\PosetComposition}[2]{
    \,
    \begin{tikzpicture}[scale=0.25]
        \ifnum#1=1 \fill[black!75] (0,0.5) rectangle (0.5,1); \fi
        \ifnum#2=1 \fill[black!75] (0.5,0) rectangle (1,0.5); \fi
        \draw[thin](0,0.5) rectangle (0.5,1);
        \draw[thin](0.5,0) rectangle (1,0.5);
    \end{tikzpicture}
    \,
}

\newcommand{\PosetMatrices}{\mathcal{PM}}

\newcommand{\SG}[1]{\textcolor{black}{#1}}

\begin{document}

\title{Operad Structure of Poset Matrices}

\author{Arnauld Mesinga Mwafise$^{b}$, Gi-Sang Cheon$^{a, b}$, Hong Joon Choi$^{b}$, Samuele Giraudo$^{c}$\\
{\footnotesize $^a$ \textit{Department of Mathematics, Sungkyunkwan
University (SKKU), Suwon 16419, Rep. of Korea}}\\
{\footnotesize $^{b}$ \textit{Applied Algebra and Optimization
Research Center, SKKU, Suwon 16419, Rep. of Korea}}\\
 {\footnotesize $^{c}$ \textit{Universit\'e du Qu\'ebec \'a Montr\'eal (UQAM), LACIM, Pavillon Pr\'esident-Kennedy, Montr\'eal, H2X 3Y7, Canada
}}\\
{\footnotesize arnauldmesinga@gmail.com, gscheon@skku.edu, h.j.choi@skku.edu, and giraudo.samuele@uqam.ca
}}

\thanks{This work was partially supported by the National Research Foundation of Korea (NRF) Grant funded by the Korean Government (MSIP)(2016R1A5A1008055).
 G.-S. Cheon was also partially supported by the NRF-2019R1A2C1007518.}
\date{}
\maketitle

\providecommand{\msc}[1]
{
  \small	
  \textbf{\textit{MSC classes---}} #1

}

\providecommand{\keywords}[1]
{
  \small	
  \textbf{\textit{Keywords---}} #1

}

\begin{abstract}
This paper examines operad structures derived from poset matrices by formulating a set of new construction rules for poset matrices. In this direction, eleven different partial composition operations will be introduced as the basis for the construction of poset matrices of any given size by extending the combinatorial setting of species of structures to poset matrices. Three of these partial composition operations are shown to define an operad structure for poset matrices. The structural properties of poset matrices and their duals are then studied based on their associated operad constructions.
\end{abstract}

\msc{06A06, 06A07, 15A36, 18D50}

\section{Introduction}

{\it Operads} \cite{mendez,may} are algebraic structures modeling the composition of operators and which also encode multilinear structures.  In most cases, it involves the possible ways of gluing together two or more disjoint structures of the same type to form a new structure of that same type, which can lead to significantly new algebraic and enumerative combinatorial interpretations. In this sense, two operators $x$ and $y$ can be composed at the $i^{th}$ position by grafting the output of $y$ on the input of $x$ to obtain the new operator $x\circ_i y.$ In addition, the inputs of $x$ can also be permuted for the case of the symmetric operads. More formally, a \textit{set operad} is a collection $\mathcal{O}:=\bigsqcup\limits_{n\ge 1}\mathcal{O}(n)$ together with partial composition maps $\circ_i$ such that
\begin{equation}\label{operad}
\circ_i:\mathcal{O}(n)\times \mathcal{O}(m)\to \mathcal{O}(n+m-1),\qquad n,m\geq 1,i\in [n],
\end{equation}
and for any  $x\in\mathcal{O}(n),y\in\mathcal{O}(m),z\in\mathcal{O}(k)$ the following properties must be satisfied:
\begin{itemize}
\item[{\rm (i)}] $\left(x\circ_i y\right)\circ_{i+j-1}z =x\circ_i(y\circ_j z)$,\; $i\in[n],j\in[m]$;
\item[{\rm (ii)}] $\left(x\circ_i y\right)\circ_{j+m-1} z = \left(x\circ_j z\right)\circ_i y$, \; $i<j\in[n]$;
\item[{\rm (iii)}] $\id \circ_1 x = x\circ_i \id = x$\; $i\in[n],$ with the identity element $id\in\mathcal{O}(1)$ called the unit of $\mathcal{O}.$
\end{itemize}

Operads derived from finite posets represented by their Hasse diagrams which takes the form of a transitively reduced directed acyclic graph has been recently studied by Fauvet et. al. \cite{fauvet} and Giraudo \cite{giraudo,giraudoAS}. The importance of posets and operads in the fields of combinatorics and computer science has been well studied for over four decades from a broad range of perspectives. In this paper, we are interested in the operad structures of finite posets in terms of their (0,1)-matrices which are generally called poset matrices. A survey of relevant research work shows that poset matrices has not yet been studied in relation to operads. This has served as a motivation to reexamine the poset operads obtained from the paper in \cite{fauvet} in terms of poset matrices. The results from this paper are expected to open new research directions of working with poset matrices from a combinatorial and an algebraic context which could not be previously achieved.

More specifically, let $P = (X,\pq)$ be a partially ordered set (poset) on a finite or an infinite set $X$. From now on, it will be assumed that $X=[n]=\{1,\ldots,n\}$ or $X={\mathbb N}$. We say that a partial order $\pq$ on $X$ is \emph{natural} if $x\pq y$ implies $x\leq y$. As pointed out by Dean and Keller in \cite{Dean}, every partial ordering
of a finite set is isomorphic to a natural partial ordering. This is a consequence
of Szpielrajn's Theorem which states that every partial ordering of a set $X$ may be refined to a total ordering. These natural ordered posets on $X$ are in bijection with $|X|\times |X|$ binary (possibly infinite) lower triangular matrices with ones on the main diagonal, i.e., {\it unit} lower triangular  matrices that contain no $(\begin{smallmatrix}1&1\\0&1\end{smallmatrix})$ submatrix whose upper right entry is on the main diagonal (see \cite{Bevan}). We note that these binary lower triangular matrices $(a_{ij})$ satisfy transitivity of their entries i.e., if $a_{ij}=1$ and $a_{jk} = 1$ then $a_{ik}=1$. Moreover, a (0,1)-triangular matrix with entries 1s in the main diagonal is clearly reflexive and antisymmetric in the sense that $a_{ii}=1$ for all $i\in [n]$ and if $a_{ij}=1$ then $a_{ji}=0$ respectively. In this sense, these matrices are called {\textit{poset matrices}}.  Therefore, such a unit (0,1)-triangular matrix is a poset matrix if it is transitive.
For $n\times n$ poset matrices $A$ and $B$, $A$ is said to be {\it permutation equivalent} to $B$ if there is a permutation matrix $Q$ such that $A=Q^TBQ$. Thus, the number of non-isomorphic posets on $n$ set $[n]$ is equal to the number of all $n\times n$ poset matrices up to permutation.
\noindent It is well known that a poset is connected if the Hasse diagram of its directed graph is connected. For the purpose of this work, we make the following definition categorizing poset matrices in terms of their connectivity with respect to their corresponding subposet matrix structures as follows.
\begin{definition}\label{disconnectMAT}
{\rm Let $A=[a_{i,j}]$ be an $n\times n$ poset matrix. If there exist at least one subposet matrix of $A$ defined on a  non-empty subset $\alpha$ of $[n]$ such that for each $k\in \alpha,$  $a_{i,k}=0$ and $a_{k,j}=0$ whenever $i,j\notin\alpha,$ then we refer to $A$ as a \textit{disconnected poset matrix.} Otherwise, we say that $A$ is a \textit{connected poset matrix}.}
\end{definition}

The rest of the paper is organized as follows. Section 2 is devoted to the construction  of operad structures by using poset matrices. In addition, several new structural properties of poset matrices arising from their operad constructions will be highlighted. Finally, in Section 3, further structural properties  of poset matrix operads are reexamined  in terms of their dual poset matrices.

\section{Structural properties of operads from poset matrices}

We begin by introducing our main approach of constructing poset matrices by explicitly outlining the partial compositions $\circ_i$ defined on the combinatorial species settings  of poset matrices. Some of these partial compositions results to an operad structure on poset matrices. We denote by $\mathcal{PM}(n)$ the set of all $n\times n$ poset matrices and let $$\mathcal{PM}=\bigsqcup\limits_{n\ge 1}\mathcal{PM}(n).$$

For the construction of operads on the species $\mathcal{PM}$ of poset matrices, we shall use some notations for an $n\times n$ matrix $A=[a_{ij}]$. Let $\alpha=\{i_1,\ldots,i_h\}$ and $\beta=\{j_1,\ldots,j_k\}$ be nonempty subsets of $[n]$. Then $A[\alpha\mid\beta]:=A[i_1,\ldots,i_h\mid j_1,\ldots,j_k]$
denotes the $h\times k$ submatrix of $A$ obtained from $A$ by taking rows $i_1,\ldots,i_h$ and columns $j_1,\ldots,j_k$. If $\alpha=\beta$, we simplify the notation to $A[\alpha]$.

For $i\in[n]$, let $\alpha_{i}=\{1,\ldots,i-1\}$ and $\beta_{n-i}=\{i+1,\ldots,n\}$ where $\alpha_{1},\beta_0$ are empty sets. For brevity, set
\begin{eqnarray}\label{e:notation}
  A_{11}:=A[\alpha_{i}],\;\; A_{22}:=A[\beta_{n-i}],\;\; A_{21}:=A[\beta_{n-i}|\alpha_{i}],\;\;A_{(i)}:=A[i|\alpha_{i}],\;\;{\rm and}\;\; A^{(i)}:=A[\beta_{n-i}|i],
\end{eqnarray}
then $A$ can be written as
\begin{eqnarray}\label{e:A}
A=\begin{pmatrix}
A_{11}&\vline&\mathbb{O}&\vline&\mathbb{O}\\
\hline
A_{(i)}&\vline &a_{ii}&\vline &\mathbb{O}\\
\hline
A_{21}&\vline&A^{(i)}&\vline&A_{22}
\end{pmatrix},
\end{eqnarray}
where $a_{ii}=1$, $A_{(i)}=(a_{i,1},\ldots,a_{i,i-1})$ and $A^{(i)}=\begin{pmatrix}
a_{i+1,i}\\\vdots\\a_{n,i}
 \end{pmatrix}.$
As usual, we also denote the $s\times t$ matrix of all zeroes by $\mathbb{O}_{s,t}$ and the $s\times t$ matrix of all ones by $\mathbb{J}_{s,t}$, whenever the size can be understood from these two context we simply denote it by $\mathbb{O}$ and $\mathbb{J}$ respectively.

Now let $A\in\mathcal{PM}(n)$ and $B\in\mathcal{PM}(m)$. For an $i\in[n]$, define $A\circ_i B$ to be a (0,1)-matrix of order $n+m-1$ obtained from $A$ in \eref{e:A} by substituting $a_{ii}$ with the $m\times m$ matrix $B$, $A_{(i)}$ with a $m\times (i-1)$ (0,1)-matrix $U_i$, and $A^{(i)}$ with a $(n-i)\times m$ (0,1)-matrix $V_i$, respectively as follows:
\begin{eqnarray}\label{e:partial}
A\circ_i B=\begin{pmatrix}
A_{11}&\vline &\mathbb{O}&\vline&\mathbb{O}\\
\hline
U_i&\vline &B&\vline&\mathbb{O}\\
\hline
A_{21}&\vline &V_i&\vline&A_{22}\\
\end{pmatrix},
\end{eqnarray}
where $U_i$ and $V_i$ are some (0,1)-matrices of sizes $m\times (i-1)$ and $(n-i)\times m$, respectively. Moreover, $A_{11}$, $A_{21}$ or $A_{22}$ may be vacuous by virtue of having no rows or no columns.

Clearly, $A\circ_i B$ is a lower triangular matrix with ones on the main diagonal and $A_{11},B, A_{22}$ are poset matrices. In the following theorem, the submatrices $U_i$ and $V_i$ of $A\circ_i B$ are determined for which $A\circ_i B$ is a poset matrix. From now on, we denote the all-ones vector of length $m$ by ${\mathbbm{1}}_m=(1,\ldots,1)$, and the Kronecker product $\otimes$ of an $m\times n$ matrix $A$ and a $p\times q$ matrix $B$ is the $mp\times nq$ matrix defined by $A\otimes B=[a_{ij}B]$. Thus,
\begin{eqnarray}\label{e:UV}
{\mathbbm{1}}_m^T\otimes A_{(i)}=\begin{pmatrix} A_{(i)}\\\vdots\\A_{(i)}\end{pmatrix}\;\;\;{\text{and}}\;\;\;
{\mathbbm{1}}_{m}\otimes A^{(i)}=\left(A^{(i)}\;\cdots\;A^{(i)}\right).
\end{eqnarray}
are $m\times (i-1)$ (0,1)-matrix and $(n-i)\times m$ (0,1)-matrix, respectively.

\begin{theorem}\label{thm1} Let $A\in{\mathcal PM}(n)$,$B\in{\mathcal PM}(m),$ and let $C$ be the matrix defined as follows:
\begin{itemize}
\item[{\rm(a)}] If $U_i=\mathbb{J}$, $V_i=\mathbb{J}$, $A_{21}=\mathbb{J},$ then
\SG{$C=A \PosetInsertion{1}{1}{1}_i B$}.
\item[{\rm(b)}] If  $U_i=\mathbb{O}$, $V_i=\mathbb{O}$, $A_{21}=\mathbb{J},$ then
\SG{$C=A \PosetInsertion{0}{1}{0}_i B$}.
\item[{\rm(c)}] If $U_i=\mathbb{J}$, $V_i=\mathbb{O}$, $A_{21}=\mathbb{J},$ then
\SG{$C=A \PosetInsertion{1}{1}{0}_i B$}.
\item[{\rm(d)}] If $U_i=\mathbb{O}$, $V_i=\mathbb{J}$, $A_{21}=\mathbb{J},$ then
\SG{$C=A \PosetInsertion{0}{1}{1}_i B$}.
\item[{\rm(e)}] If $U_i=\mathbb{O}$, $V_i=\mathbb{O}$, $A_{21}=\mathbb{O},$ then
\SG{$C=A \PosetInsertion{0}{0}{0}_i B$}.
\item[{\rm(f)}] If $U_i=\mathbb{O}$, $V_i=\mathbb{J}$, $A_{21}=\mathbb{O},$ then
\SG{$C=A \PosetInsertion{0}{0}{1}_i B$}.
\item[{\rm(g)}]If $U_i=\mathbb{J}$, $V_i=\mathbb{O}$, $A_{21}=\mathbb{O},$ then
\SG{$C=A \PosetInsertion{1}{0}{0}_i B$}.
\end{itemize}
In all cases $C$ is a poset matrix.
\end{theorem}
\begin{proof} Let $i\in[n]$, and assume that $U_i,V_i$ and $A_{21}$ can either be the zero matrix $\mathbb{O}$ or the all ones matrix $\mathbb{J}.$ Clearly, all the $7$ cases (a)--(g) of assigning $\mathbb{O}$ and $\mathbb{J}$ to  $U_i,V_i$ and $A_{21}$  avoid the forbidden pattern for non-transitivity of the entries of a poset matrix whereby $U=\mathbb{J}$, $V=\mathbb{J}$, and $A_{21}=\mathbb{O}.$
\end{proof}

More generally, we have the following theorem.

\begin{theorem}\label{thm1} Let $A\in\mathcal{PM}(n)$, $B\in\mathcal{PM}(m)$ and  $\square_i:{\cal PM}(n)\times{\cal PM}(m)\rightarrow {\cal PM}(n+m-1)$ be partial composition maps  for  $i\in[n]$ defined by
\begin{align}\label{eq1}
A\square_i B =
\left[\begin{array}{c|c|c}
A_{11} & \mathbb{O}& \mathbb{O}\\
\hline
{\mathbbm{1}}_m^T\otimes A_{(i)} & B & \mathbb{O}\\
\hline
A_{21} & {\mathbbm{1}}_{m}\otimes A^{(i)} & A_{22}
\end{array}\right].
\end{align}
\end{theorem}
Then the partial compositions~$\square_i$ endows the species of poset matrices  $\SG{\PosetMatrices}$ with an operad structure.
\begin{proof} Let $M=A\square_i B$. First we show that $M$ is a poset matrix in ${\cal PM}(n+m-1)$ for any $i\in[n]$. Clearly, $M$ is a unit lower triangular matrix of order $n+m-1$. Noticing the matrices in \eref{e:UV},
the transitivity of $M$ immediately follows from the transitivity of $A$ and $B$. Thus $M$ is a poset matrix.

Now we show that $\square_i$ satisfies three axioms for operad. Let $A\in {\cal PM}(n)$, $B\in {\cal PM}(m)$, $C\in {\cal PM}(k)$, and assume that $i\in[n]$ and $j\in[m]$ are arbitrarily chosen.

(i) By definition we obtain
\begin{eqnarray*}
(A\square_i B)\square_{i+j-1}C=M\square_{i+j-1}C=\begin{pmatrix}
M_{11}&\vline &\mathbb{O}&\vline&\mathbb{O}\\
\hline
{\mathbbm{1}}_k^T\otimes M_{(i+j-1)}&\vline &C&\vline&\mathbb{O}\\
\hline
M_{21}&\vline &{\mathbbm{1}}_k\otimes M^{(i+j-1)}&\vline&M_{22}\\
\end{pmatrix},
\end{eqnarray*}
which is equivalent to
\begin{eqnarray*}\label{e:partial-1}
\begin{pmatrix}
A[\alpha_{i-1}] &\vline& \mathbb{O}&\vline & \mathbb{O}&\vline & \mathbb{O}&\vline & \mathbb{O}\\
\hline
	{\mathbbm{1}}_{j-1}^T\otimes A_{(i)}&\vline &B[\alpha_{j-1}]&\vline & \mathbb{O}&\vline & \mathbb{O} &\vline& \mathbb{O}\\
	\hline
	{\mathbbm{1}}_{k}^T\otimes A_{(i)}&\vline & {\mathbbm{1}}_{k}^T\otimes B_{(j)}&\vline & C &\vline& \mathbb{O}&\vline & \mathbb{O}\\
	\hline
	{\mathbbm{1}}_{m-j}^T\otimes A_{(i)}&\vline & B[\beta_{m-j}|\alpha_{j-1}]&\vline & {\mathbbm{1}}_{k}\otimes B^{(j)}&\vline & B[\beta_{m-j}]&\vline & \mathbb{O}\\
	\hline
	A[\beta_{n-i}|\alpha_{i-1}]&\vline & {\mathbbm{1}}_{j-1}\otimes A^{(j)} &\vline & {\mathbbm{1}}_{k}\otimes A^{(i)} &\vline& {\mathbbm{1}}_{m-j}\otimes A^{(i)}&\vline & A[\beta_{n-i}]
\end{pmatrix}
\end{eqnarray*}

\begin{eqnarray*}
=\begin{pmatrix}
A_{11}&\vline &\mathbb{O}&\vline&\mathbb{O}\\
\hline
{\mathbbm{1}}_{m+k-1}^T\otimes A_{(i)}&\vline &B\square_jC&\vline&\mathbb{O}\\
\hline
A_{21}&\vline &{\mathbbm{1}}_{m+k-1}\otimes A^{(i)}&\vline&A_{22}
\end{pmatrix}=A\square_i (B\square_jC).
\end{eqnarray*}
Thus $(A\square_i B)\square_{i+j-1}C=A\square_i (B\square_jC).$
\vskip.5pc
(ii) To show that $(A\square_i B) \square_{j+m-1} C = (A \square_j C) \square_i B$, let $1\leq i < j \leq n$ and
	\[\gamma=\{i+1,\ldots, j-1\}.
	\]
	Then
	\begin{align*}
		(A\square_i B) \square_{j+m-1} C &=
		\begin{pmatrix}
			A[\alpha_i] & \vline & \mathbb{O} & \vline & \mathbb{O}\\
			\hline
			{\mathbbm{1}}_m^T\otimes A_{(i)} & \vline & B & \vline & \mathbb{O}\\
			\hline
			A[\beta_{n-i}|\alpha_{i}] & \vline & {\mathbbm{1}}_{m}\otimes A^{(i)} & \vline & A[\beta_{n-i}]
	    \end{pmatrix} \square_{j+m-1}C\\
		&=
		\begin{pmatrix}
			A[\alpha_i] & \vline & \mathbb{O} & \vline & \mathbb{O} & \vline & \mathbb{O} & \vline & \mathbb{O}\\
			\hline
			{\mathbbm{1}}_m^T\otimes A_{(i)} & \vline & B & \vline & \mathbb{O} & \vline & \mathbb{O} & \vline & \mathbb{O}\\
			\hline
			A[\gamma|\alpha_i] & \vline & {\mathbbm{1}}_m\otimes A[\gamma|i] & \vline & A[\gamma] & \vline & \mathbb{O} & \vline & \mathbb{O}\\
			\hline
			{\mathbbm{1}}_k^T\otimes A[j|\alpha_i] & \vline & {\mathbbm{1}}_k^T \otimes {\mathbbm{1}}_m \otimes A_{ji} & \vline & {\mathbbm{1}}_k^T\otimes A[j|\gamma] & \vline & C & \vline & \mathbb{O}\\
			\hline
			A[\beta_{n-j}|\alpha_i] & \vline & {\mathbbm{1}}_m \otimes A[\beta_{n-j}|i] & \vline & A[\beta_{n-j}|\gamma] & \vline & {\mathbbm{1}}_k\otimes A^{(j)} & \vline & A[\beta_{n-j}]
		\end{pmatrix}\\
		&=\begin{pmatrix}
			A[\alpha_j] & \vline & \mathbb{O} & \vline & \mathbb{O}\\
			\hline
			{\mathbbm{1}}_k^T\otimes A_{(j)} & \vline & C & \vline & \mathbb{O}\\
			\hline
			A[\beta_{n-j}|\alpha_j] & \vline & {\mathbbm{1}}_k\otimes A^{(j)} & \vline & A[\beta_{n-j}]
	    \end{pmatrix} \square_i B\\
		&=(A \square_j C) \square_i B.
			\end{align*}

(iii) Since $[1]\square_1 A = A\square_i [1] = A$ for $A\in \mathcal{PM}(n)$ and $i\in[n]$, the $1\times 1$ poset matrix $[1]$ acts as the identity in ${\cal O}(1)$.

\noindent Therefore $\square_i$ endows the species of poset matrices $\cal PM$ with an operad structure.
\end{proof}

The following corollary is an immediate consequence of Theorem \ref{thm1}.
\begin{corollary}\label{rc} For $A\in{\cal PM}(n)$, let $u_i=A[i|\alpha_i]$ and $v_i=A[\beta_i|i]$ where $\alpha_i=\{1,\ldots,i-1\}$ and $\beta_i=\{i+1,\ldots,n\}$.
\begin{itemize}
\item  Assume that $A[\beta_i|\alpha_i]=\mathbb{J}_{n-i,i-1}.$
\begin{itemize}
\item[{\rm (a)}] If $u_i=(1,\ldots,1)$ and $v_i=(1,\ldots,1)^T$ then $A\square_i B =
\SG{A \PosetInsertion{1}{1}{1}_i B}$.
\item[{\rm (b)}] If $u_i=(0,\ldots,0)$ and $v_i=(0,\ldots,0)^T$ then $A\square_i B =
\SG{A \PosetInsertion{0}{1}{0}_i B}$.
\item[{\rm (c)}] If $u_i=(1,\ldots,1)$ and $v_i=(0,\ldots,0)^T$ then $A\square_i B =
\SG{A \PosetInsertion{1}{1}{0}_i B}$.
\item[{\rm (d)}] If $u_i=(0,\ldots,0)$ and $v_i=(1,\ldots,1)^T$ then $A\square_i B =
\SG{A \PosetInsertion{0}{1}{1}_i B}$.
\end{itemize}
\item  Assume that $A[\beta_i|\alpha_i]=\mathbb{O}_{n-i,i-1}.$
\begin{itemize}
\item[{\rm (e)}] If $u_i=(1,\ldots,1)$ and $v_i=(1,\ldots,1)^T$ then $A\square_i B =
\SG{A \PosetInsertion{0}{0}{0}_i B}$.
\item[{\rm (f)}] If $u_i=(0,\ldots,0)$ and $v_i=(1,\ldots,1)^T$ then $A\square_i B =
\SG{A \PosetInsertion{0}{0}{1}_i B}$.
\item[{\rm (g)}] If $u_i=(1,\ldots,1)$ and $v_i=(0,\ldots,0)^T$ then $A\square_i B =
\SG{A \PosetInsertion{1}{0}{0}_i B}$.
\end{itemize}
\end{itemize}
\end{corollary}
%

A {\it minimal} element of a poset $P=(X,\leqslant)$ is an element $a$ such that $a >x$ for no $x\in X$; {\it maximal} elements are defined dually. Any finite subset $X$ of a poset $P$ has minimal and maximal elements. In the following lemma, this notion is interpreted from the perspective of the poset matrix.

\begin{lemma}\label{minmax} Let $A\in {\cal PM}(n)$ be a poset matrix associated with the poset $P$ naturally labelled in $[n]$.
\begin{itemize}
\item[{\rm(a)}] An element $i\in P$ is minimal if and only if $i=1$ or $A_{(i)}$ is the zero vector.
\item[{\rm(b)}]	An element $i\in P$ is maximal if and only if $i=n$ or $A^{(i)}$ is the zero vector.
\end{itemize}
\end{lemma}
\begin{proof} (a) By definition, an element $i\in P$ is minimal if and only if there is no $j\in P$ such that $j<i$.
It follows that $i\in P$ is a minimal element if and only if $i=1$ or $a_{ij}=0$ for every $j$ with $i>j$, i.e. $A_{(i)}$ is the zero vector.

(b) Similarly, an element $i\in P$ is maximal if and only if there is no $j\in P$ such that $i<j$. It follows that
$i\in P$ is a maximal element if and only if $i=n$ or $a_{ji}=0$ for every $j$ with $i<j$, i.e. $A^{(i)}$ is the zero vector.
\end{proof}

\begin{example} {\rm Consider the poset matrix $A\in{\cal PM}(4)$ associated to the poset $P$:}
\begin{eqnarray*}
 A=\begin{pmatrix}
	1 & 0 & 0 & 0 \\
	1 & 1 & 0 & 0 \\
	0 & 0 & 1 & 0 \\
	1 & 0 & 1 & 1\\
\end{pmatrix}\quad\Leftrightarrow\quad
P: \quad
\begin{minipage}[c]{.20\textwidth}
\begin{tikzpicture}
  [scale=.7,auto=center,every node/.style={circle,fill=blue!20}]
\node(d1) at (9,1) {1};
\node(d6) at (9,3)  {2};
\node(d3) at (11,1)  {3};
\node(d5) at (11,3)  {4};
\draw (d6) -- (d1);
\draw (d3) -- (d5);
\draw (d1) -- (d5);
 \end{tikzpicture}
\end{minipage}
\end{eqnarray*}
{\rm Since $A_{(3)}=[0\;0]$ and $A^{(2)}=[0\;0]^T$, it follows from Lemma \ref{minmax} that the minimal elements of $P$ are 1, 3 and the maximal elements of
 $P$ are 2, 4, which are agree with the Hasse diagram of $P$ as shown above.}
\end{example}

\noindent In the following theorems, we similarly obtain two other operad structures of
$\SG{\PosetMatrices}$.

\noindent Firstly, we define:

${\rm Min}_i(A,B)=[v_1,\ldots,v_m]$ as the $(n-i)\times m$ matrix whose $j$th column $v_j$ is given by
\begin{equation}\label{min}
	v_j=
	\begin{cases}
		A^{(i)} & \text{if $j\in[m]$ is a minimal element of the poset associated to $B$},\\
		\mathbb{O} & \text{otherwise}.
	\end{cases}
\end{equation}

\begin{theorem}\label{thm3} Let $A\in\mathcal{PM}(n)$, $B\in\mathcal{PM}(m)$ $ and  \;\SG{\PosetComposition{0}{1}_i}:\SG{\PosetMatrices}(n)
\times \SG{\PosetMatrices}(m)\rightarrow \SG{\PosetMatrices}(n+m-1)$ be partial composition maps for $i\in[n]$ defined by
\begin{eqnarray}\label{eq2}
	A \SG{\PosetComposition{0}{1}_i} B =\begin{pmatrix}
		A_{11}&\vline & \mathbb{O} &\vline& \mathbb{O}\\
		\hline
	{\mathbbm{1}}_m^T\otimes A_{(i)}&\vline &B&\vline&\mathbb{O}\\
		\hline
	A_{21}&\vline & {\rm Min}_i(A,B) & \vline&A_{22}\\
	\end{pmatrix}.
\end{eqnarray}
Then the partial compositions~$\SG{\PosetComposition{0}{1}_i}$ endows the species of poset matrices  $\SG{\PosetMatrices}$ with an operad structure.
\end{theorem}
\begin{proof} Let $A\in {\cal PM}(n)$, $B\in {\cal PM}(m)$, and $C\in {\cal PM}(k)$. Assume that $i\in[n]$ and $j\in[m]$ are arbitrarily chosen. Note that ${\rm Min}_i(A,B)$ defined in (\ref{min}) represents the minimal elements of the poset ${\cal P}_X$ associated to the poset matrix $X$ for $X=A$ or $B$. Thus, using the similar argument used in the proof of Theorem \ref{thm1} yields
\begin{align*}
&(A \SG{\PosetComposition{0}{1}_i} B)
\SG{\PosetComposition{0}{1}_{i + j - 1}}
C\\&=\begin{pmatrix}
A[\alpha_{i}] &\vline&\mathbb{O}&\vline & \mathbb{O}&\vline & \mathbb{O}&\vline & \mathbb{O}\\
\hline
	{\mathbbm{1}}_{j-1}^T\otimes A_{(i)}&\vline &B[\alpha_{j}]&\vline & \mathbb{O}&\vline & \mathbb{O} &\vline& \mathbb{O}\\
	\hline
	{\mathbbm{1}}_{k}^T\otimes A_{(i)}&\vline & {\mathbbm{1}}_{k}^T\otimes B_{(j)}&\vline & C &\vline& \mathbb{O}&\vline & \mathbb{O}\\
	\hline
	{\mathbbm{1}}_{m-j}^T\otimes A_{(i)}&\vline & B[\beta_{m-j}|\alpha_{j}]&\vline & {\rm Min}_j(B,C)&\vline & B[\beta_{m-j}]&\vline & \mathbb{O}\\
	\hline
	A[\beta_{n-i}|\alpha_{i}]&\vline & {\rm Min}_i(A,B)[\beta_{n-i}|\alpha_{j}] &\vline & {\rm Min}_i(A,C) &\vline& {\rm Min}_i(A,B)[\beta_{n-i}|\beta_{m-j}]&\vline & A[\beta_{n-i}]
\end{pmatrix}\\
&=\begin{pmatrix}
A_{11}&\vline &\mathbb{O}&\vline&\mathbb{O}\\
\hline
{\mathbbm{1}}_{m+k-1}^T\otimes A_{(i)}&\vline &B\SG{\PosetComposition{0}{1}_j} C&\vline&\mathbb{O}\\
\hline
A_{21}&\vline &{\rm Min}_i(A,B \SG{\PosetComposition{0}{1}_j}C)&\vline&A_{22}
\end{pmatrix}\\
&=A\SG{\PosetComposition{0}{1}_i} (B\SG{\PosetComposition{0}{1}_j} C).
\end{align*}

\noindent To show that $(A\SG{\PosetComposition{0}{1}_i} B) \SG{\PosetComposition{0}{1}_{j +
m - 1}}
C = (A \SG{\PosetComposition{0}{1}_j} C) \SG{\PosetComposition{0}{1}_i} B$,
let $1\leq i < j \leq n$ and
\[\gamma=\{i+1,\ldots, j-1\}.\]
Then
\begin{align*}
	(A \SG{\PosetComposition{0}{1}_i} B) \SG{\PosetComposition{0}{1}_{j + m - 1}}
	C &=
	\begin{pmatrix}
		A[\alpha_i] & \vline & \mathbb{O} & \vline & \mathbb{O}\\
		\hline
		{\mathbbm{1}}_m^T\otimes A_{(i)} & \vline & B & \vline & \mathbb{O}\\
		\hline
		A[\beta_{n-i}|\alpha_i] & \vline & {\rm Min}_i(A,B) & \vline & A[\beta_{n-i}]
    \end{pmatrix} \SG{\PosetComposition{0}{1}_{j + m - 1}} C\\
	&=
	\begin{pmatrix}
		A[\alpha_i] & \vline & \mathbb{O} & \vline & \mathbb{O} & \vline & \mathbb{O} & \vline & \mathbb{O}\\
		\hline
		{\mathbbm{1}}_m^T\otimes A_{(i)} & \vline & B & \vline & \mathbb{O} & \vline & \mathbb{O} & \vline & \mathbb{O}\\
		\hline
		A[\gamma|\alpha_i] & \vline & {\rm Min}_i(A,B)[\gamma|\alpha_{m+1}] & \vline & A[\gamma] & \vline & \mathbb{O} & \vline & \mathbb{O}\\
		\hline
		{\mathbbm{1}}_k^T\otimes A[j|\alpha_i] & \vline & {\rm Min}_i(A,B)[j|\alpha_{m+1}] & \vline & {\mathbbm{1}}_k^T\otimes A[j|\gamma] & \vline & C & \vline & \mathbb{O}\\
		\hline
		A[\beta_{n-j}|\alpha_i] & \vline & {\rm Min}_i(A,B)[\beta_{n-j}|\alpha_{m+1}] & \vline & A[\beta_{n-j}|\gamma] & \vline & {\rm Min}_j(A,C) & \vline & A[\beta_{n-j}]
	\end{pmatrix}\\
	 &=
	\begin{pmatrix}
		A[\alpha_j] & \vline & \mathbb{O} & \vline & \mathbb{O}\\
		\hline
		{\mathbbm{1}}_k^T\otimes A_{(j)} & \vline & C & \vline & \mathbb{O}\\
		\hline
		A[\beta_{n-j}|\alpha_j] & \vline & {\rm Min}_j(A,C) & \vline & A[\beta_{n-j}]
    \end{pmatrix} \SG{\PosetComposition{0}{1}_i} B\\
	&=(A \SG{\PosetComposition{0}{1}_j} C) \SG{\PosetComposition{0}{1}_i} B.
	\end{align*}

 Moreover, $[1] \SG{\PosetComposition{0}{1}_1} A = A \SG{\PosetComposition{0}{1}_i} [1] = A$
 for $A\in \Op(n)$ and $i\in[n]$. Thus the $1\times 1$ poset matrix $[1]$ acts as the
identity in ${\cal O}(1)$. Therefore $\SG{\PosetComposition{0}{1}_i}$ endows the species of
poset matrices $\cal PM$ with an operad structure. \end{proof}

\noindent Secondly, we define:

${\rm Max}_i(A,B)=[u_1,u_2,\ldots,u_m]^T$ as the $m\times (i-1)$ matrix with $j$th row $u_j$ given by
\begin{equation*}
	{\bf u}_j=
	\begin{cases}
		A_{(i)} & \text{if $j\in[m]$ is a maximal element of the poset associated to $B$},\\
		\mathbb{O} & \text{otherwise}.
	\end{cases}
\end{equation*}
\begin{theorem}\label{thm4} Let $A\in\mathcal{PM}(n)$, $B\in\mathcal{PM}(m)$ and  \; $\SG{\PosetComposition{1}{0}_i}:\SG{\PosetMatrices}(n)
\times\SG{\PosetMatrices}(m)\rightarrow \SG{\PosetMatrices}(n+m-1)$ be the partial composition maps for $i\in[n]$ defined by
\begin{eqnarray}\label{eq3}
	A\SG{\PosetComposition{1}{0}_i} B =\begin{pmatrix}
		A_{11}&\vline & \mathbb{O} &\vline& \mathbb{O}\\
		\hline
	{\rm Max}_i(A,B)&\vline &B&\vline&\mathbb{O}\\
		\hline
	A_{21}&\vline & {\mathbbm{1}}_{m}\otimes A^{(i)} & \vline&A_{22}\\
	\end{pmatrix}.
\end{eqnarray}
Then the partial compositions~$\SG{\PosetComposition{1}{0}_i}$ endows the species of poset matrices  $\SG{\PosetMatrices}$ with an operad structure.
\end{theorem}
 \begin{proof} It follows from Theorem \ref{thm3} and (2) of Theorem \ref{dualthm} in Section 4.
 \end{proof}

\begin{theorem}\label{thmMinMaxPoset} Let $A\in\mathcal{PM}(n)$, $B\in\mathcal{PM}(m)$ $ and  \;\SG{\PosetComposition{1}{1}_i}:\SG{\PosetMatrices}(n)
\times \SG{\PosetMatrices}(m)\rightarrow \SG{\PosetMatrices}(n+m-1)$ be the partial composition maps for $i\in[n]$ defined by
\begin{eqnarray}\label{eq22}
	A \SG{\PosetComposition{1}{1}_i} B =\begin{pmatrix}
		A_{11}&\vline & \mathbb{O} &\vline& \mathbb{O}\\
		\hline
	{\rm Max}_i(A,B)&\vline &B&\vline&\mathbb{O}\\
		\hline
	A_{21}&\vline & {\rm Min}_i(A,B) & \vline&A_{22}\\
	\end{pmatrix}.
\end{eqnarray}
Then the partial compositions~$\SG{\PosetComposition{1}{1}_i}$ results to poset matrices for all $i.$

\end{theorem}

\begin{proof}


Assume that  $U_i={\rm Max}_i(A,B)$ and $V_i={\rm Min}_i(A,B)$ in (\ref{e:partial}) results to (\ref{eq22}). Let $u_{k,j}$ and $v_{y,z}$ denote an entries in the submatrices $U_i$ and $V_i$ respectively. Consider the  following possibilities of assigning entries to $U_i$ and $V_i$.
\begin{enumerate}
\item By definition  of ${\rm Max}_i(A,B)$ each index  $k$ associated to the maximal element of the poset matrix $B$ has as entry $u_{k,j}\in\{0,1\}.$ On the other hand, the entry $v_{y,k}=0$ for each $k$ associated to the maximal element of $B.$ Therefore, the transitivity condition of the entry indexed by the pair $(y,j)$ in the submatrix $A_{21}$ cannot be violated.

\item By definition  of ${\rm Min}_i(A,B)$ each index  $z$ associated to the minimal element of the poset matrix $B$ has as entry $v_{y,z}\in\{0,1\}.$ On the other hand, the entry $u_{z,j}=0$ for each $z$ associated to the minimal element of $B.$ Therefore, the transitivity condition of the entry indexed by the pair $(y,j)$ in the submatrix $A_{21}$ cannot be violated.

\item  Consider the case in which the index $k$ is not a maximal element associated to $B$ in the submatrix $U_i=(u_{k,j})$ or when the index $z$ is not a minimal element associated to $B$ in the submatrix  $V_i=(v_{y,z}).$ In both cases the entries $u_{k,j}$ and $v_{y,z}$ are zero. The transitivity of the output poset matrix of $A \SG{\PosetComposition{1}{1}_i} B $ cannot be violated.

\end{enumerate}

Since $A_{11}, B$ and $A_{22}$ are poset matrices in (\ref{eq22})  and the cases (1),(2) and (3) possible ways of assigning entries to the submatrices $U_i$ and $V_i$ do not violate the transitivity condition of poset matrices.  Thus, the output of $A \SG{\PosetComposition{1}{1}_i} B $ is a poset matrix for all $i.$

\end{proof}

\begin{example}\label{minmaxops} \text{Let} $A\in {\cal PM}(4)$, $B\in {\cal PM}(3)$ \text{and} $C\in {\cal PM}(2)$  \text{as follows}:
 \begin{align*}
	A =
	\left[\begin{array}{cccc}
		1 & 0 & 0 & 0 \\
		1 & 1 & 0 & 0 \\
		1 & 0 & 1 & 0 \\
		1 & 1 & 1 & 1
	\end{array}\right],\quad
	B =
	\left[\begin{array}{ccc}
		1 & 0 & 0 \\
		1 & 1 & 0 \\
		1 & 0 & 1
			\end{array}\right],\quad
	C =
	\left[\begin{array}{cc}
		1 & 0\\
		1 & 1
              \end{array}\right].
\end{align*}
Then
\begin{align*}
	A \square_2 B =
	\left[\begin{array}{c|ccc|cc}
		1 & 0 & 0 & 0 & 0 & 0 \\
		\hline
		1 & 1 & 0 & 0 & 0 & 0 \\
		1 & 1 & 1 & 0 & 0 & 0 \\
		1 & 1 & 0 & 1 & 0 & 0 \\
		\hline
		1 & 0 & 0 & 0 & 1 & 0 \\
		1 & 1 & 1 & 1 & 1 & 1
	\end{array}\right],
\end{align*}
\begin{align*}
	A \SG{\PosetComposition{0}{1}_2} B =
	\left[\begin{array}{c|ccc|cc}
		1 & 0 & 0 & 0 & 0 & 0 \\
		\hline
		1 & 1 & 0 & 0 & 0 & 0 \\
		1 & 1 & 1 & 0 & 0 & 0 \\
		1 & 1 & 0 & 1 & 0 & 0 \\
		\hline
		1 & 0 & 0 & 0 & 1 & 0 \\
		1 & 1 & 0 & 0 & 1 & 1
	\end{array}\right],\quad
	A \SG{\PosetComposition{1}{0}_2} B =
	\left[\begin{array}{c|ccc|cc}
		1 & 0 & 0 & 0 & 0 & 0 \\
		\hline
		0 & 1 & 0 & 0 & 0 & 0 \\
		1 & 1 & 1 & 0 & 0 & 0 \\
		1 & 1 & 0 & 1 & 0 & 0 \\
		\hline
		1 & 0 & 0 & 0 & 1 & 0 \\
		1 & 1 & 1 & 1 & 1 & 1
	\end{array}\right]   \quad  \text{and}   \quad  A \SG{\PosetComposition{1}{1}_2} B =
	\left[\begin{array}{c|ccc|cc}
		1 & 0 & 0 & 0 & 0 & 0 \\
		\hline
		0 & 1 & 0 & 0 & 0 & 0 \\
		1 & 1 & 1 & 0 & 0 & 0 \\
		1 & 1 & 0 & 1 & 0 & 0 \\
		\hline
		1 & 0 & 0 & 0 & 1 & 0 \\
		1 & 1 & 0 & 0 & 1 & 1
	\end{array}\right].
\end{align*}
\end{example}

\begin{remark}
Using the poset matrices $A, B,$ and $C$ as defined in example \ref{minmaxops}, it can be verified that when  $i=2,j=3$ the partial composition~$\SG{\PosetComposition{1}{1}}$ is not an operad as illustrated below.

\begin{align*}\left(A\SG{\PosetComposition{1}{1}_i} B\right)\SG{\PosetComposition{1}{1}_{i+j-1}}C =\left[\begin{array}{ccc|cc|cc}
		1 & 0 & 0 & 0 & 0 & 0 &0\\
		0 & 1 & 0 & 0 & 0 & 0&0 \\
		1 & 1 & 1 & 0 & 0 & 0 &0\\
                   \hline
		0 & 0 & 0 & 1 & 0 & 0& 0\\
	         1 & 1 & 0 & 1 & 1 & 0&0 \\
                  	\hline
		1 & 0 & 0 & 0 & 0 & 1&0\\
                   1 & 1 & 0 & 0 & 0 & 1&1
	\end{array}\right]\end{align*}

and  \begin{align*}A\SG{\PosetComposition{1}{1}_i}(B\SG{\PosetComposition{1}{1}_j} C)=\left[\begin{array}{c|cccc|cc}
		1 & 0 & 0 & 0 & 0 & 0 &0\\
                   \hline
		0 & 1 & 0 & 0 & 0 & 0&0 \\
		1 & 1 & 1 & 0 & 0 & 0 &0\\
                  0 & 0 & 0 & 1 & 0 & 0& 0\\
	         1 & 1 & 0 & 1 & 1 & 0&0 \\
                  	\hline
		1 & 0 & 0 & 0 & 0 & 1&0\\
                   1 & 1 & 0 & 1 & 0 & 1&1
	\end{array}\right].\end{align*}
\end{remark}
Since as  $\left(A\SG{\PosetComposition{1}{1}_i} B\right)\SG{\PosetComposition{1}{1}_{i+j-1}}C\neq$$  A\SG{\PosetComposition{1}{1}_i}(B\SG{\PosetComposition{1}{1}_j} C)$, it follows that  $\SG{\PosetComposition{1}{1}}$ is not an operad by the nested associativity axiom in \ref{operad}.
\begin{theorem}\label{dpm}
The partial composition operation $A\;\scalebox{0.7}{$\square$}_i\;B$ forms a disconnected poset matrix for all $i$ if and only if $A$ is a disconnected poset matrix.
\end{theorem}

\begin{lemma}\label{dpmops}
Let $D_n$ denote the set of all disconnected poset matrices of order $n\times n.$ Then,
$D_n= \SG{P_{\PosetInsertion{0}{1}{0}_i}} \cup \SG{P_{\PosetInsertion{0}{0}{0}_i}} \cup
\SG{P_{\PosetInsertion{0}{0}{1}_i}} \cup \SG{P_{\PosetInsertion{1}{0}{0}_i}}$, where
$\SG{P_{\PosetInsertion{0}{1}{0}_i}}$, $\SG{P_{\PosetInsertion{0}{0}{1}_i}}$,
$\SG{P_{\PosetInsertion{1}{0}{0}_i}}$, and $\SG{P_{\PosetInsertion{0}{0}{0}_i}}$ correspond
to the set of all poset matrices derived from the partial composition operations
$\SG{\PosetInsertion{0}{1}{0}_i}$, $\SG{\PosetInsertion{0}{0}{1}_i}$,
$\SG{\PosetInsertion{1}{0}{0}_i}$, and $\SG{\PosetInsertion{0}{0}{0}_i}$ respectively.
\end{lemma}
\begin{proof}
The poset matrices derived from
\SG{$P_{\PosetInsertion{0}{0}{1}_i}$},
\SG{$P_{\PosetInsertion{1}{0}{0}}$},
and \SG{$P_{\PosetInsertion{0}{0}{0}_i}$}
are of the forms
$$\begin{pmatrix}
A_{11}&\vline &\mathbb{O}&\vline&\mathbb{O}\\
\hline
\mathbb{O}&\vline &B&\vline&\mathbb{O}\\
\hline
\mathbb{O}&\vline &\mathbb{O}&\vline&A_{22}\\
\end{pmatrix},\begin{pmatrix}
A_{11}&\vline &\mathbb{O}&\vline&\mathbb{O}\\
\hline
\mathbb{O}&\vline &B&\vline&\mathbb{O}\\
\hline
\mathbb{O}&\vline &\mathbb{J}&\vline&A_{22}\\
\end{pmatrix},\;{\rm and}\;
\begin{pmatrix}
A_{11}&\vline &\mathbb{O}&\vline&\mathbb{O}\\
\hline
\mathbb{J}&\vline &B&\vline&\mathbb{O}\\
\hline
\mathbb{O}&\vline &\mathbb{O}&\vline&A_{22}\\
\end{pmatrix},$$
which can all be expressible as
$$\left(
\begin{array}{c|c}
G &\mathbb{O} \\
\hline
\mathbb{O} & H
\end{array}
\right)$$ where $G$ and $H$ are poset matrices.

Similarly, the partial composition operation $\SG{\PosetInsertion{0}{1}{0}_i}$
generates only disconnected poset matrices since $A_{21}=\mathbb{F}$ and the entries of the submatrices  $U_i$ and $V_i$ in \ref{e:partial} are all $\rm 0$ for this case.
\end{proof}

The possible candidates for the input poset matrices to obtain a particular output poset matrix for the  partial composition operations described in (\ref{eq1}) can be determined as follows.

\begin{theorem}\label{inputoutputposet}
Let $C_{n+m-1}$ be a poset matrix of size $(n+m-1)\times (n+m-1)$ labeled on the set $X$, and let $A$ and  $B$ be poset matrices of sizes $n\geq 2$ and $m\geq 2$ respectively. Then $C_{n+m-1}=A\;\scalebox{0.7}{$\square$}_i\;B$ if and only if $A$ and $B$ are subposet matrices of $C_{n+m-1}$.
\end{theorem}
\begin{proof}

Suppose that $C_{n+m-1}=A\;\scalebox{0.7}{$\square$}_i\;B,$ then $B$ will always be a subposet matrix of
$C_{n+m-1}$ by the construction of poset matrices using $\scalebox{0.7}{$\square$}_i.$ Since the entries of $A_{(i)}$ and $A^{(i)}$ in (\ref{eq1}) are determined by the deleted row and column entries at position $i$ of the poset matrix $A$ for the operation $\scalebox{0.7}{$\square$}_i$, it will therefore form a subposet matrix without the label $i.$ Therefore, the relationship between $A$ and $C_{n+m-1}$ is established using their unlabeled poset matrices.

 The reverse direction of the proof is obvious from the definition of $\scalebox{0.7}{$\square$}_i$. Also, the sizes of the matrices $A$ and $B$ must be greater or equal to $2$ since if it is $1$ then it is the identity matrix and the problem becomes trivial.
\end{proof}

\begin{example} {\rm Let $A\;\scalebox{0.8}{$\square$}_2 \;B=C_{n+m-1}$ be represented as
$$\begin{blockarray}{cccc}
1 & 2&3\\
\begin{block}{(ccc)c}
 1 & 0&0  &1\\
 0 & 1&0 &2\\
 1 & 1&1 &3 \\
\end{block}
\end{blockarray}\scalebox{0.7}{$\square$}_2\;\;
\begin{blockarray}{ccc}
4 & 5 \\
\begin{block}{(cc)c}
 1 & 0&4 \\
 1 & 1&5\\
\end{block}
\end{blockarray}=
\begin{blockarray}{ccccc}
1&4 & 5 &3\\
\begin{block}{(cccc)c}
 1 & 0 & 0&0 &1\\
 0 & 1 & 0&0 &4\\
 0 & 1 & 1&0 &5\\
 1 & 1 & 1&1 &3\\
\end{block}
\end{blockarray}$$
In the above example, we can deduce that $A$ is isomorphic to the subposet matrix $C_4[1,4,3]$ of $C_4.$}
\end{example}

 It is shown \cite{riordanposets} that every $n\times n$ binary Pascal matrix $P_n=[b_{ij}]$ where $b_{ij}={i\choose j}$ (mod $2$) is the poset matrix. The corresponding poset is called the Pascal poset ${\cal P}_n$. If $n=2^k$ then the Pascal poset ${\cal P}_n$ is isomorphic to the {\it $k$-dimensional Boolean lattice} ${\mathbb B}_{k}$ consisting of all subsets of a $k$-element set ordered by inclusion. As noted in the paper \cite{Csaba}, every finite poset is isomorphic to a subposet of a sufficiently large Boolean lattice. It turns out that every finite poset matrix is a submatrix of the infinite binary Pascal matrix.

\begin{corollary} Every $n\times n$ binary Pascal matrix $P_n$ can be represented by $P_n=P_2\;\scalebox{0.7}{$\square$}_2\;P_n(1|1)$, where $P_n(1|1)$ is the $(n-1)\times (n-1)$ matrix obtained from $P_n$ by deleting first row and first column.
\end{corollary}

\subsection{Operad construction of $3\times 3$ and $4\times 4$  poset matrices}
 We begin by generating all the possible connected poset matrices in the set $\mathcal {PM}(3)$ using the partial composition operation $\scalebox{0.7}{$\square_i$}$ defined in (\ref{eq1})  as follows:

\begin{enumerate}
\item  $\begin{blockarray}{ccc}
1 & 2 \\
\begin{block}{(cc)c}
 1 & 0  &1\\
 1 & 1 &2 \\
\end{block}
\end{blockarray}\scalebox{0.7}{$\square$}_1\;\;\begin{blockarray}{ccc}
1 & 2 \\
\begin{block}{(cc)c}
 1 & 0&1 \\
 1 & 1&2\\
\end{block}
\end{blockarray}=
\begin{blockarray}{cccc}
1&2 & 3 \\
\begin{block}{(ccc)c}
 1 & 0 & 0&1 \\
 1 & 1 & 0&2 \\
 1 & 1 & 1&3 \\
\end{block}
\end{blockarray}\quad\Leftrightarrow\quad$
\begin{minipage}[c]{.20\textwidth}
\begin{tikzpicture}
  [scale=.7,auto=center,every node/.style={circle,fill=blue!20}]
\node(d1) at (9,1.7) {2};
\node(d6) at (9,3)  {3};
\node(d7) at (9,0.5) {1};
\draw (d6) -- (d1);
\draw (d1) -- (d7);

 \end{tikzpicture}

\end{minipage}

\item $\begin{blockarray}{ccc}
1 & 2 \\
\begin{block}{(cc)c}
 1 & 0  &1\\
 1 & 1 &2 \\
\end{block}
\end{blockarray}\scalebox{0.7}{$\square$}_1\;\;\begin{blockarray}{ccc}
1 & 2 \\
\begin{block}{(cc)c}
 1 & 0&1 \\
 0 & 1&2\\
\end{block}
\end{blockarray}=
\begin{blockarray}{cccc}
1&2& 3 \\
\begin{block}{(ccc)c}
 1 & 0 & 0&1 \\
 0 & 1 & 0&2 \\
 1 & 1 & 1&3 \\
\end{block}
\end{blockarray}\quad\Leftrightarrow\quad$
\begin{minipage}[c]{.20\textwidth}
\begin{tikzpicture}
  [scale=.7,auto=center,every node/.style={circle,fill=blue!20}]
\node(d1) at (10,1.7) {3};
\node(d6) at (11,0.5)  {2};
\node(d7) at (9,0.5) {1};
\draw (d6) -- (d1);
\draw (d1) -- (d7);

 \end{tikzpicture}

\end{minipage}

\item $\begin{blockarray}{ccc}
1 & 2 \\
\begin{block}{(cc)c}
 1 & 0  &1\\
 1 & 1 &2 \\
\end{block}
\end{blockarray}\scalebox{0.7}{$\square$}_2\;\;\begin{blockarray}{ccc}
1 & 2 \\
\begin{block}{(cc)c}
 1 & 0&1 \\
 0 & 1&2\\
\end{block}
\end{blockarray}=
\begin{blockarray}{cccc}
1&2 & 3 \\
\begin{block}{(ccc)c}
 1 & 0 & 0&1 \\
 1 & 1 & 0&2 \\
 1 & 0 & 1&3 \\
\end{block}
\end{blockarray}\quad\Leftrightarrow\quad$
\begin{minipage}[c]{.20\textwidth}
\begin{tikzpicture}
  [scale=.7,auto=center,every node/.style={circle,fill=blue!20}]
\node(d1) at (10,0.5) {1};
\node(d6) at (11,1.7)  {3};
\node(d7) at (9,1.7) {2};
\draw (d6) -- (d1);
\draw (d1) -- (d7);

 \end{tikzpicture}

\end{minipage}

\end{enumerate}

Similarly, the $2$ disconnected poset matrices of size $3\times 3$ in the set set $\mathcal {PM}(3)$  can be obtained as follows:

\begin{enumerate}
\item  $\begin{blockarray}{ccc}
1 & 2 \\
\begin{block}{(cc)c}
 1 & 0  &1\\
 0 & 1 &2 \\
\end{block}
\end{blockarray}\scalebox{0.7}{$\square$}_1\;\;\begin{blockarray}{ccc}
1 & 2 \\
\begin{block}{(cc)c}
 1 & 0&1 \\
 0 & 1&2\\
\end{block}
\end{blockarray}=
\begin{blockarray}{cccc}
1&2 & 3 \\
\begin{block}{(ccc)c}
 1 & 0 & 0&1 \\
 0 & 1 & 0&2 \\
 0 & 0 & 1&3 \\
\end{block}
\end{blockarray}\quad\Leftrightarrow\quad$
\begin{minipage}[c]{.20\textwidth}
\begin{tikzpicture}
  [scale=.7,auto=center,every node/.style={circle,fill=blue!20}]
\node(d1) at (10,1.7) {2};
\node(d6) at (11,1.7)  {3};
\node(d7) at (9,1.7) {1};

 \end{tikzpicture}

\end{minipage}

\item  $\begin{blockarray}{ccc}
1 & 2 \\
\begin{block}{(cc)c}
 1 & 0  &1\\
 0 & 1 &2 \\
\end{block}
\end{blockarray}\scalebox{0.7}{$\square$}_1\;\;\begin{blockarray}{ccc}
1 & 2 \\
\begin{block}{(cc)c}
 1 & 0&1 \\
 1 & 1&2\\
\end{block}
\end{blockarray}=
\begin{blockarray}{cccc}
1&2 & 3 \\
\begin{block}{(ccc)c}
 1 & 0 & 0&1 \\
 1 & 1 & 0&2 \\
 0 & 0 & 1&3 \\
\end{block}
\end{blockarray}\quad\Leftrightarrow\quad$
\begin{minipage}[c]{.20\textwidth}
\begin{tikzpicture}
  [scale=.7,auto=center,every node/.style={circle,fill=blue!20}]
\node(d1) at (9,1.3) {2};
\node(d6) at (9,2.5)  {1};
\node(d7) at (10.2,1.3) {3};
\draw (d6) -- (d1);

 \end{tikzpicture}

\end{minipage}

\end{enumerate}

The $10$ connected poset matrices of size $4\times 4$ in the set $\mathcal {PM}(4)$  can be obtained from the partial composition operations defined by (\ref{eq1}) and (\ref{eq3}) as follows:

\begin{enumerate}
\item $\begin{blockarray}{cccc}
1 & 2&3\\
\begin{block}{(ccc)c}
 1 & 0&0  &1\\
 1 & 1&0 &2\\
 1 & 1&1 &3 \\
\end{block}
\end{blockarray}\scalebox{0.7}{$\square$}_3\;\;
\begin{blockarray}{ccc}
1 & 2 \\
\begin{block}{(cc)c}
 1 & 0&1 \\
 0 & 1&2\\
\end{block}
\end{blockarray}=
\begin{blockarray}{ccccc}
1&2 & 3 &4\\
\begin{block}{(cccc)c}
 1 & 0 & 0&0 &1\\
 1 & 1 & 0&0 &2\\
 1 & 1 & 1&0 &3\\
 1 & 1 & 0&1 &4\\
\end{block}
\end{blockarray}\quad\Leftrightarrow\quad$
\begin{minipage}[c]{.20\textwidth}
\begin{tikzpicture}
  [scale=.7,auto=center,every node/.style={circle,fill=blue!20}]
\node(d1) at (10,4) {2};
\node(d6) at (11,5)  {4};
\node(d7) at (9,5) {3};
\node(d8) at (10,2.5) {1};
\draw (d6) -- (d1);
\draw (d1) -- (d7);
\draw (d1) -- (d8);

 \end{tikzpicture}

\end{minipage}

\item $\begin{blockarray}{cccc}
1 & 2&3\\
\begin{block}{(ccc)c}
 1 & 0&0  &1\\
 0 & 1&0 &2\\
 1 & 1&1 &3 \\
\end{block}
\end{blockarray}\scalebox{0.7}{$\square$}_3\;\;\begin{blockarray}{ccc}
1 &2 \\
\begin{block}{(cc)c}
 1 & 0&1 \\
 1 & 1&2\\
\end{block}
\end{blockarray}=
\begin{blockarray}{ccccc}
1&2 & 3 &4\\
\begin{block}{(cccc)c}
 1 & 0 & 0&0 &1\\
 0 & 1 & 0&0 &2\\
 1 & 1 & 1&0 &3\\
 1 & 1 & 1&1 &4\\
\end{block}
\end{blockarray}\quad\Leftrightarrow\quad$
\begin{minipage}[c]{.20\textwidth}
\begin{tikzpicture}
  [scale=.7,auto=center,every node/.style={circle,fill=blue!20}]

\node(d4) at (10,2) {4};
\node(d1) at (10,0.5) {3};
\node(d6) at (11,-0.3)  {2};
\node(d7) at (9,-0.3) {1};
\draw (d6) -- (d1);
\draw (d1) -- (d7);
\draw (d1) -- (d4);

 \end{tikzpicture}

\end{minipage}

\item $\begin{blockarray}{cccc}
1 & 2&3\\
\begin{block}{(ccc)c}
 1 & 0&0  &1\\
 1 & 1&0 &2\\
 1 & 1&1 &3 \\
\end{block}
\end{blockarray}\scalebox{0.7}{$\square$}_1\;\;
\begin{blockarray}{ccc}
1 & 2 \\
\begin{block}{(cc)c}
 1 & 0&1 \\
 1 & 1&2\\
\end{block}
\end{blockarray}=
\begin{blockarray}{ccccc}
1&2 & 3 &4\\
\begin{block}{(cccc)c}
 1 & 0 & 0&0 &1\\
 1 & 1 & 0&0 &2\\
 1 & 1 & 1&0 &3\\
 1 & 1 & 1&1 &4\\
\end{block}
\end{blockarray}\quad\Leftrightarrow\quad$
\begin{minipage}[c]{.20\textwidth}
\begin{tikzpicture}
  [scale=.7,auto=center,every node/.style={circle,fill=blue!20}]
\node(d1) at (10,4) {3};
\node(d6) at (10,3)  {2};
\node(d7) at (10,5) {4};
\node(d8) at (10,2) {1};
\draw (d6) -- (d1);
\draw (d1) -- (d7);
\draw (d6) -- (d8);

 \end{tikzpicture}

\end{minipage}

\item $
\begin{blockarray}{ccc}
1 & 2 \\
\begin{block}{(cc)c}
 1 & 0&1 \\
 1 & 1&2\\
\end{block}
\end{blockarray}
\scalebox{0.7}{$\square$}_1\;\;\begin{blockarray}{cccc}
1 & 2&3\\
\begin{block}{(ccc)c}
 1 & 0&0  &1\\
 1 & 1&0 &2\\
 1 & 0&1 &3 \\
\end{block}
\end{blockarray}
=\begin{blockarray}{ccccc}
1&2 & 3 &4\\
\begin{block}{(cccc)c}
 1 & 0 & 0&0 &1\\
 1 & 1 & 0&0 &2\\
 1 & 0 & 1&0 &3\\
 1 & 1 & 1&1 &4\\
\end{block}
\end{blockarray}\quad\Leftrightarrow\quad$
\begin{minipage}[c]{.20\textwidth}
\begin{tikzpicture}
  [scale=.7,auto=center,every node/.style={circle,fill=blue!20}]

\node[rectangle] (c1) at (6.2,3) {2};
  \node[rectangle] (c2) at (7,2) {1};
  \node[rectangle] (c4) at (7.8,3) {3};
  \node[rectangle] (c8) at (7,4)  {4};

      \draw (c1) -- (c2);
        \draw (c2) -- (c4);
    \draw (c8) -- (c4);
   \draw (c8) -- (c1);

 \end{tikzpicture}

\end{minipage}

\item $\begin{blockarray}{cccc}
1 & 2&3\\
\begin{block}{(ccc)c}
 1 & 0&0  &1\\
 1 & 1&0 &2\\
 1 & 0&1 &3 \\
\end{block}
\end{blockarray}\scalebox{0.7}{$\square$}_1\;\;
\begin{blockarray}{ccc}
1& 2 \\
\begin{block}{(cc)c}
 1 & 0&1 \\
 0 & 1&2\\
\end{block}
\end{blockarray}=
\begin{blockarray}{ccccc}
1&2 & 3 &4\\
\begin{block}{(cccc)c}
 1 & 0 & 0&0 &1\\
 0 & 1 & 0&0 &2\\
 1 & 1 & 1&0 &3\\
 1 & 1 & 0&1 &4\\
\end{block}
\end{blockarray}\quad\Leftrightarrow\quad$
\begin{minipage}[c]{.20\textwidth}
\begin{tikzpicture}
  [scale=.7,auto=center,every node/.style={circle,fill=blue!20}]
\node(d1) at (10,5) {3};
\node(d8) at (10,3) {1};
\node(d4) at (12,5) {4};
\node(d5) at (12,3) {2};
\draw (d4) -- (d5);
\draw (d1) -- (d8);
\draw (d1) -- (d5);
\draw (d4) -- (d8);

 \end{tikzpicture}

\end{minipage}

\item $\begin{blockarray}{cccc}
1 & 2&3\\
\begin{block}{(ccc)c}
 1 & 0&0  &1\\
 1 & 1&0 &2\\
 1 & 0&1 &3 \\
\end{block}
\end{blockarray}\scalebox{0.7}{$\square$}_2\;\;\begin{blockarray}{ccc}
1 &2 \\
\begin{block}{(cc)c}
 1 & 0&1 \\
 0 & 1&2\\
\end{block}
\end{blockarray}
=\begin{blockarray}{ccccc}
1&2 & 3 &4\\
\begin{block}{(cccc)c}
 1 & 0 & 0&0 &1\\
 1 & 1 & 0&0 &2\\
 1 & 0 & 1&0 &3\\
 1 & 0 & 0&1 &4\\
\end{block}
\end{blockarray}\quad\Leftrightarrow\quad$
\begin{minipage}[c]{.20\textwidth}
\begin{tikzpicture}
  [scale=.7,auto=center,every node/.style={circle,fill=blue!20}]

\node[rectangle] (c1) at (5.5,3) {2};
  \node[rectangle] (c2) at (7,1) {1};
  \node[rectangle] (c4) at (7,3) {3};
   \node[rectangle] (c5) at (8.5,3) {4};
      \draw (c5) -- (c2);
      \draw (c1) -- (c2);
        \draw (c2) -- (c4);

 \end{tikzpicture}

\end{minipage}

\item $\begin{blockarray}{cccc}
1 & 2&3\\
\begin{block}{(ccc)c}
 1 & 0&0  &1\\
 0 & 1&0 &2\\
 1 & 1&1 &3 \\
\end{block}
\end{blockarray}\scalebox{0.7}{$\square$}_1\;\;
\begin{blockarray}{ccc}
1 & 2 \\
\begin{block}{(cc)c}
 1 & 0&1 \\
 0 & 1&2\\
\end{block}
\end{blockarray}=
\begin{blockarray}{ccccc}
1&2 & 3 &4\\
\begin{block}{(cccc)c}
 1 & 0 & 0&0 &1\\
 0 & 1 & 0&0 &2\\
 0 & 0 & 1&0 &3\\
 1 & 1 & 1&1 &4\\
\end{block}
\end{blockarray}\quad\Leftrightarrow\quad$
\begin{minipage}[c]{.20\textwidth}
\begin{tikzpicture}
  [scale=.7,auto=center,every node/.style={circle,fill=blue!20}]

\node(d4) at (10,-0.3) {2};
\node(d1) at (10,1) {4};
\node(d6) at (11,-0.3)  {1};
\node(d7) at (9,-0.3) {3};
\draw (d6) -- (d1);
\draw (d1) -- (d7);
\draw (d1) -- (d4);

 \end{tikzpicture}

\end{minipage}

\item $\begin{blockarray}{cccc}
1 & 2&3\\
\begin{block}{(ccc)c}
 1 & 0&0  &1\\
 0 & 1&0 &2\\
 1 & 1&1 &3 \\
\end{block}
\end{blockarray}\scalebox{0.7}{$\square$}_1\;\;\begin{blockarray}{ccc}
1 &2 \\
\begin{block}{(cc)c}
 1 & 0&1 \\
 1 & 1&2\\
\end{block}
\end{blockarray}
=\begin{blockarray}{ccccc}
1&2 & 3 &4\\
\begin{block}{(cccc)c}
 1 & 0 & 0&0 &1\\
 1 & 1 & 0&0 &2\\
 0 & 0 & 1&0 &3\\
 1 & 1 & 1&1 &4\\
\end{block}
\end{blockarray}\quad\Leftrightarrow\quad$
\begin{minipage}[c]{.20\textwidth}
\begin{tikzpicture}
  [scale=.7,auto=center,every node/.style={circle,fill=blue!20}]

\node[rectangle] (c1) at (6.2,3) {3};
  \node[rectangle] (c2) at (7.8,1.7) {1};
  \node[rectangle] (c4) at (7.8,3) {2};
  \node[rectangle] (c8) at (7,4)  {4};

      \draw (c4) -- (c2);
    \draw (c8) -- (c4);
   \draw (c8) -- (c1);

 \end{tikzpicture}

\end{minipage}

\item $\begin{blockarray}{cccc}
1 & 2&3\\
\begin{block}{(ccc)c}
 1 & 0&0  &1\\
 1 & 1&0 &2\\
 1 & 0&1 &3 \\
\end{block}
\end{blockarray}\scalebox{0.7}{$\square$}_2\;\;\begin{blockarray}{ccc}
1 & 2 \\
\begin{block}{(cc)c}
 1 & 0&1 \\
 1 & 1&2\\
\end{block}
\end{blockarray}
=\begin{blockarray}{ccccc}
1&2 & 3 &4\\
\begin{block}{(cccc)c}
 1 & 0 & 0&0 &1\\
 1 & 1 & 0&0 &2\\
 1 & 1 & 1&0 &3\\
 1 & 0 & 0&1 &4\\
\end{block}
\end{blockarray}\quad\Leftrightarrow\quad$
\begin{minipage}[c]{.20\textwidth}
\begin{tikzpicture}
  [scale=.7,auto=center,every node/.style={circle,fill=blue!20}]

\node[rectangle] (c1) at (6.2,2.5) {2};
  \node[rectangle] (c2) at (7,1.5) {1};
  \node[rectangle] (c4) at (7.8,2.5) {4};
  \node[rectangle] (c8) at (6.2,3.7)  {3};

      \draw (c1) -- (c2);
       \draw (c4) -- (c2);
   \draw (c8) -- (c1);

 \end{tikzpicture}

\end{minipage}

\item $\begin{blockarray}{cccc}
1 & 2&3\\
\begin{block}{(ccc)c}
 1 & 0&0  &1\\
 1 & 1&0 &2\\
 1 & 0&1 &3 \\
\end{block}
\end{blockarray}\SG{\PosetComposition{1}{0}_2}\;\;
\begin{blockarray}{ccc}
1 & 2 \\
\begin{block}{(cc)c}
 1 & 0&1 \\
 1 & 1&2\\
\end{block}
\end{blockarray}=
\begin{blockarray}{ccccc}
1&2 & 3 &4\\
\begin{block}{(cccc)c}
 1 & 0 & 0&0 &1\\
 0 & 1 & 0&0 &2\\
 1 & 1 & 1&0 &3\\
 1 & 0 & 0&1 &4\\
\end{block}
\end{blockarray}\quad\Leftrightarrow\quad$
\begin{minipage}[c]{.20\textwidth}
\begin{tikzpicture}
  [scale=.7,auto=center,every node/.style={circle,fill=blue!20}]

\node(d1) at (9,1) {2};
\node(d6) at (9,3)  {3};
\node(d3) at (11,1)  {1};
\node(d5) at (11,3)  {4};
\draw (d6) -- (d1);
\draw (d3) -- (d5);
\draw (d1) -- (d5);

 \end{tikzpicture}

\end{minipage}

\end{enumerate}

Similarly, the $6$ disconnected poset matrices of size $4\times 4$ in the set $\mathcal {PM}(4)$   can be generated as follows.

\begin{enumerate}

\item $\begin{blockarray}{cccc}
1 & 2&3\\
\begin{block}{(ccc)c}
 1 & 0&0& 1\\
 1 & 1&0 &2\\
  0 & 0&1&3\\
\end{block}
\end{blockarray}\scalebox{0.7}{$\square$}_3\;\;
\begin{blockarray}{ccc}
1 & 2 \\
\begin{block}{(cc)c}
1 & 0&1\\
1 & 1&2\\
\end{block}
\end{blockarray}=
\begin{blockarray}{ccccc}
1&2 & 3 &4\\
\begin{block}{(cccc)c}
 1 & 0 & 0&0&1 \\
 1 & 1 & 0&0&2 \\
 0 & 0 & 1&0&3 \\
 0 & 0 & 1&1&4 \\
\end{block}
\end{blockarray}\quad\Leftrightarrow\quad$
\begin{minipage}[c]{.20\textwidth}
 \begin{tikzpicture}
\node [rectangle] (bb1) at (3,5) {2};
\node [rectangle] (bb2) at (3,3)  {1};
\node [rectangle]  (bb3) at (4,5) {4};
\node [rectangle] (bb4) at (4,3)  {3};
\node [use](aa17) at (5,4) {$.$};
\draw (bb2) -- (bb1);
\draw (bb3) -- (bb4);
\end{tikzpicture}

\end{minipage}

\item $\begin{blockarray}{cccc}
1 & 2&3\\
\begin{block}{(ccc)c}
1 & 0&0&1  \\
 0 & 1&0&2 \\
 0 & 0&1&3\\
\end{block}
\end{blockarray}\scalebox{0.7}{$\square$}_1\;\;
\begin{blockarray}{ccc}
1 & 2 \\
\begin{block}{(cc)c}
1 & 0&1\\
0 & 1&2\\
\end{block}
\end{blockarray}=
\begin{blockarray}{ccccc}
1&2 & 3 &4\\
\begin{block}{(cccc)c}
  1 & 0 & 0&0&1 \\
 0 & 1 & 0&0&2 \\
 0 & 0 & 1&0&3 \\
 0 & 0 & 0&1&4 \\
\end{block}
\end{blockarray}\quad\Leftrightarrow\quad$
\begin{minipage}[c]{.20\textwidth}
\begin{tikzpicture}
  [scale=.7,auto=center,every node/.style={circle,fill=blue!20}]
\node(d1) at (9,1) {1};
\node(d6) at (10,1)  {2};
\node(d7) at (11,1) {3};
\node(d8) at (12,1) {4};
\end{tikzpicture}
 \end{minipage}

\item $\begin{blockarray}{cccc}
1 & 2&3\\
\begin{block}{(ccc)c}
 1 & 0&0&1  \\
 0 & 1&0&2\\
   0 & 1&1&3\\
\end{block}
\end{blockarray}\scalebox{0.7}{$\square$}_1\;\;
\begin{blockarray}{ccc}
1 & 2 \\
\begin{block}{(cc)c}
1 & 0& 1\\
0 & 1& 2\\
\end{block}
\end{blockarray}=
\begin{blockarray}{ccccc}
1&2 & 3 &4\\
\begin{block}{(cccc)c}
 1 & 0 & 0&0&1 \\
 0 & 1 & 0&0&2 \\
 0 & 0 & 1&0&3 \\
 0 & 0 & 1&1&4 \\
\end{block}
\end{blockarray}\quad\Leftrightarrow\quad$
\begin{minipage}[c]{.20\textwidth}
\begin{tikzpicture}
  [scale=.7,auto=center,every node/.style={circle,fill=blue!20}]
\node(d1) at (9,1) {3};
\node(d6) at (9,2.5)  {4};
\node(d7) at (10,1) {1};
\node(d7) at (11,1) {2};
\draw (d6) -- (d1);
\end{tikzpicture}
 \end{minipage}

\item $\begin{blockarray}{cccc}
1 & 2&3\\
\begin{block}{(ccc)c}
  1 & 0&0&1 \\
 1 & 1&0&2 \\
  0 & 0&1&3\\
\end{block}
\end{blockarray}\scalebox{0.7}{$\square$}_1\;\;
\begin{blockarray}{ccc}
1 & 2 \\
\begin{block}{(cc)c}
1 & 0&1\\
1 & 1&2\\
\end{block}
\end{blockarray}=
\begin{blockarray}{ccccc}
1&2 & 3 &4\\
\begin{block}{(cccc)c}
  1 & 0 & 0&0 &1\\
 1 & 1 & 0&0&2 \\
 1 & 1 & 1&0&3 \\
 0 & 0 & 0&1&4 \\
\end{block}
\end{blockarray}\quad\Leftrightarrow\quad$
\begin{minipage}[c]{.20\textwidth}
\begin{tikzpicture}
  [scale=.7,auto=center,every node/.style={circle,fill=blue!20}]
\node(d1) at (9,1) {1};
\node(d6) at (9,2.5)  {2};
\node(d8) at (9,4)  {3};
\node(d7) at (10,1) {4};
\draw (d6) -- (d1);
\draw (d6) -- (d8);
\end{tikzpicture}
 \end{minipage}

\item $\begin{blockarray}{cccc}
1 & 2&3\\
\begin{block}{(ccc)c}
  1 & 0&0&1  \\
 0 & 1&0&2 \\
  0 & 1&1&3\\
\end{block}
\end{blockarray}\scalebox{0.7}{$\square$}_2\;\;
\begin{blockarray}{ccc}
1 & 2 \\
\begin{block}{(cc)c}
1 & 0&1\\
0 & 1&2\\
\end{block}
\end{blockarray}=
\begin{blockarray}{ccccc}
1&2 & 3 &4\\
\begin{block}{(cccc)c}
   1 & 0 & 0&0&1 \\
 0 & 1 & 0&0&2 \\
 0 & 0 & 1&0&3 \\
 0 & 1 & 1&1&4 \\
\end{block}
\end{blockarray}\quad\Leftrightarrow\quad$
\begin{minipage}[c]{.20\textwidth}
\begin{tikzpicture}
  [scale=.7,auto=center,every node/.style={circle,fill=blue!20}]
\node(d1) at (10,1.7) {4};
\node(d6) at (11,0.5)  {3};
\node(d7) at (9,0.5) {2};
\node(d88) at (12,0.5) {1};
\draw (d6) -- (d1);
\draw (d1) -- (d7);

 \end{tikzpicture}
  \end{minipage}

\item $\begin{blockarray}{cccc}
1 & 2&3\\
\begin{block}{(ccc)c}
  1 & 0&0&1  \\
 0 & 1&0&2 \\
   0 & 1&1&3\\
\end{block}
\end{blockarray}\scalebox{0.7}{$\square$}_3\;\;
\begin{blockarray}{ccc}
1 & 2 \\
\begin{block}{(cc)c}
1 & 0&1\\
 0 & 1&2\\
\end{block}
\end{blockarray}=
\begin{blockarray}{ccccc}
1&2 & 3 &4\\
\begin{block}{(cccc)c}
     1 & 0 & 0&0 &1\\
 0 & 1 & 0&0 &2\\
 0 & 1 & 1&0 &3\\
 0 & 1 & 0&1&4 \\
\end{block}
\end{blockarray}\quad\Leftrightarrow\quad$
\begin{minipage}[c]{.20\textwidth}
\begin{tikzpicture}
  [scale=.7,auto=center,every node/.style={circle,fill=blue!20}]
\node(d1) at (10,0.5) {2};
\node(d6) at (11,1.7)  {4};
\node(d7) at (9,1.7) {3};
\node(d88) at (12,0.5) {1};
\draw (d6) -- (d1);
\draw (d1) -- (d7);

 \end{tikzpicture}
   \end{minipage}

\end{enumerate}

\begin{definition}{\rm 
Let  $\text{A}=[a_{i,j}]$ be a poset matrix of size $n\times n$. If $a_{i,j}=1$ for all $i\in[n]$ whenever $i\leq j,$  then we call $\text{A}$ a totally connnected poset matrix.}
\end{definition}
\begin{definition}{\rm 
Let  $\text{A}=[a_{i,j}]$ be a poset matrix of size $n\times n.$ If $a_{i,j}=0$ for all $i\in[n]$ whenever $i\neq j,$ Then we call $\text{A}$ a totally disconnected poset matrix.}
\end{definition}


\begin{remark}{\rm 
For the purpose of clarity on the notations used in the theorems that follow, given a poset matrix $A$ of size $n\times n,$ we refer to $A_{(j)}[\{k,...,n\}\mid\{1,..,k-1\}]$ and $A^{(j)}[\{k,...,n\}\mid\{1,..,k-1\}]$ to denote the $j^{th}$ row of size $1\times(k-1)$ and the $j^{th}$ column of size $(n-k+1)\times 1$ respectively of the submatrix $A[\{k,...,n\}\mid\{1,..,k-1\}]$ derived from the matrix $A.$ If $k=2$ and $n=5$ then the set $\{k,...,n\}$ is equivalent to $\{2,3,4,5\}.$}
\end{remark}

\begin{theorem}\label{connectedsubposet} Let $A$ be a connected poset matrix of size $n\times n$ and let $\text{B}$ be a totally connected poset matrix of size $m\times m.$ Then the following holds.
\begin{enumerate}
\item  If there exist a totally connected subposet matrix $A[\alpha]$ defined on $\alpha=\{1,...,k\}$ with $1<k<n$ such that the submatrix $D:=A[\{k+1,...,n\}|\{1,...,k\}]$ satisfies the condition that $D^{(p)}=D^{(p+1)}$ for each $p\in\{1,...,k-1\}.$ Then $\text{A}\square_i \text{B}=\text{A}\square_r \text{B}$ whenever $i,r\in\alpha.$
\item If there exist a totally connected subposet matrix $A[\alpha]$ defined on $\alpha=\{k,...,n\}$ with $k\geq 2$ such that the submatrix $D:=A[\{k,...,n\}|\{1,...,k-1\}]$ satisfies the condition that $D_{(p)}=D_{(p+1)}$ for each $p\in\{1,...,n-k\}.$ Then $\text{A}\square_i \text{B}=\text{A}\square_r \text{B}$ whenever $i,r\in\alpha.$
\item   If there exist a totally connected subposet matrix $A[\alpha]$ defined on $\alpha=\{d,...,k\}$ with  $1<d<k<n$ such that the submatrix $D:=A[\{k,...,n\}|\{1,...,k-1\}]$ satisfies the condition $D_{(p)}=D_{(p+1)}$ for each $p\in\{1,...,n-k\}$ where $D_{(p)}$ is a   \textbf{1}-vector or a \textbf{0}-vector of size $1 \times(k-1).$ Then $\text{A}\square_i \text{B}=\text{A}\square_r \text{B}$ whenever $i,r\in\alpha.$
\end{enumerate}
\end{theorem}
\begin{proof}
Let $j,j+1\in\alpha.$ In (1), it suffices to show that the output poset matrix $A\square_{j} B=A\square_{j +1}B$ for $j\in\{1,...,k\}.$   Consider that the output poset matrix structure of $A\square_{j}B=E$ can be partitioned into three submatrices $E1, E2$ and $E3$ such that:
\begin{itemize}
\item $E1=E[\{1,...,k+m-1\}]$  of size $(k+m-1)\times (k+m-1).$
\item $E2=E[\{k+m,...,n+m-1\}\mid\{1,...,k+m-1\}]$ of size $(n-k)\times (k+m-1).$
\item $E3=E[\{k+m,...,n+m-1\}]$ of size $(n-k)\times(n-k).$
\end{itemize}
$E1$ can be subpartitioned such that $E[\{j,...,j+m\}]=B$, $E[\{j,...,j+m\}\mid\{1,...,j-1\} ]$, $E[\{j+m+1,....,k+m-1\}\mid\{1,...,k+m-1\}]$ and $E[\{1,...,j-1\}]$  have entries of all $1$'s that lie on and below the main diagonal of $E$ since these entries are derived from $A[\{1,...,k\}]$ which is a totally connected subposet matrix.

\noindent $E2$ can be subpartitioned such that $E[\{k+m,...,n+m-1\}\mid\{1,...,j-1\}]=A[\{k+1,...,n+m-1\}\mid\{1,...,j-1\}], E[\{k+m,...,n+m-1\}\mid\{j,...,j+m\}]={\mathbbm{1}}_{m}\otimes A^{(j)}[\{k+1,..n\}\mid\{1,..,k\}]$, $ E[\{k+m,...,n+m-1\}\mid\{j+m+1,...,k\}]=A[\{k+1,...,n+m-1\}\mid\{j+m+1,...,k]$ and $E3=A[\{k+1,...,n+m-1\}].$

In the adjacent insertion point $j+1$, similarly consider that the output poset matrix structure of $A\square_{j+1}B=F$ can be partitioned into three submatrices $F1, F2$ and $F3$ such that:
\begin{itemize}
\item $F1=F[\{1,...,k+m-1\}]$  of size $(k+m-1)\times (k+m-1).$
\item $F2=F[\{k+m,...,n+m-1\}\mid\{1,...,k+m-1\}]$ of size $(n-k)\times (k+m-1).$
\item $F3=F[\{k+m,...,n+m-1\}]$ of size $(n-k)\times(n-k).$
\end{itemize}
$F1$ can be subpartitioned such that $F[\{j+1,...,j+1+m\}]=B$, $F[\{j+1,...,j+1+m\}\mid\{1,...,j\} ]$, $F[\{j+m+2,....,k+m-1\}\mid\{1,...,k+m-1\}]$ and $F[\{1,...,j\}]$  have entries of all $1$'s that lie on and below the main diagonal of $F$ since these entries are derived from $A[\{1,...,k\}]$ which is a totally connected subposet matrix.

\noindent $F2$ can be subpartitioned such that $F[\{k+m,...,n+m-1\}\mid\{1,...,j\}]=A[\{k+1,...,n+m-1\}\mid\{1,...,j\}], F[\{k+m,...,n+m-1\}\mid\{j+1,...,j+1+m\}]={\mathbbm{1}}_{m}\otimes A^{(j+1)}[\{k+1,..n\}\mid\{1,..,k\}]$, $ F[\{k+m,...,n+m-1\}\mid\{j+m+2,...,k\}]=A[\{k+1,...,n+m-1\}\mid\{j+m+2,...,k]$ and $F3=A[\{k+1,...,n+m-1\}].$

From the matrix structure of $E$ and $F$ presented above, we get the following.

\begin{itemize}
\item $E1=F1$ since all the subpartitions of $E1$ and $F1$ are equal.
\item $E2=F2$ since as the subpartitions can be equal whenever ${\mathbbm{1}}_{m}\otimes A^{(j)}[\{k+1,..n\}\mid\{1,..,k\}]={\mathbbm{1}}_{m}\otimes A^{(j+1)}[\{k+1,..n\}\mid\{1,..,k\}].$ This condition holds when each of the columns in the submatrix $A[\{k+1,..n\}\mid\{1,..,k\}]$ have their correponding entries equal.
\item $E3=F3.$
\end{itemize}
 Thus,  $A\square_{j} B=A\square_{j +1}B$ for $j\in\{1,...,k\}.$

Let $j,j+1\in\alpha.$ In (2) it suffices to show that the output poset matrix $A\square_{j} B=A\square_{j +1}B$ for $j\in\{k,...,n\}.$  Consider that the output poset matrix structure of $A\square_{j}B=G$ can be partitioned into three submatrices $G1, G2$ and $G3$ such that:

\begin{itemize}
\item $G1=G[\{1,...,k-1\}]$ of size $(k-1)\times(k-1).$
\item  $G2=G[\{k,...,n+m-1\}\mid\{1,...,k-1\}]$ of size $(n+m-k)\times(k-1).$
\item $G3=G[\{k,...,n+m-1\}]$ of size $(n+m-k)\times(n+m-k)$.
\end{itemize}
We note that $G1=A[\{1,...,k-1\}].$ $G2$ can be subpartitioned such that  $G[\{j,...,j+m\}\mid\{1,...,k-1\}]={\mathbbm{1}}_m^T\otimes A_{(j)}[\{k,...,n\}\mid\{1,...,k-1\}]$ and $G[\{j+m+1,...,n+m-1\}\mid\{1,...,k-1\}]=A[\{j+1,...,n\}\mid\{1,...,k-1].$ $G3$ can be subpartitioned such that $G[\{j,...j+m\}]=B,$ $G[\{k,..,j-1\}]$, and $G[\{j+m+1,...,n+m-1\}]$ have entries of all $1$'s that lie on and below the main diagonal of $G$ since these entries are derived from $A[\{k,...,n\}]$ which is a totally connected subposet matrix.

In the adjacent insertion point $j+1$, similarly consider that the output poset matrix structure of $A\square_{j+1}B=H$ can be partitioned into three submatrices $H1, H2$ and $H3$ such that:
\begin{itemize}
\item $H1=H[\{1,...,k-1\}]$ of size $(k-1)\times(k-1).$
\item  $H2=H[\{k,...,n+m-1\}\mid\{1,...,k-1\}]$ of size $(n+m-k)\times(k-1).$
\item $H3=H[\{k,...,n+m-1\}]$ of size $(n+m-k)\times(n+m-k)$.
\end{itemize}
Based on the matrix structures for $G$ and $H$ we obtain the following.
\begin{itemize}
\item $H1=A[\{1,...,k-1\}]=G1.$
\item By replacing $j:=j+1$ in $G2$ and noting that ${\mathbbm{1}}_m^T\otimes A_{(j)}[\{k,...,n\}\mid\{1,...,k-1\}]={\mathbbm{1}}_m^T\otimes A_{(j+1)}[\{k,...,n\}\mid\{1,...,k-1\}]$ is satisfied  when each of the rows in the submatrix $A[\{k,..,n\}\mid\{1,..,k-1\}]$ have their correponding entries equal. Therefore whenever this condition holds $G2=H2.$
\item $G3=H3$ since as replacing $j$ in $E3$ with $j+1,$ preserves the same matrix structure in $H3.$
\end{itemize}
 Thus,  $A\square_{j} B=A\square_{j +1}B$ for each $j\in\{k,...,n\}.$

\noindent The proof of (3) follows from the arguments in (1) and (2).

\end{proof}

\begin{example} {\rm Consider the poset matrices}
\end{example}
 $$\text{A}=\begin{blockarray}{ccccc}
1 & 2&3&4\\
\begin{block}{(cccc)c}
 1 & 0&0  &0&1\\
 1 & 1&0 &0&2\\
 1 & 1&1 &0&3 \\
 1&1&0&1&4\\
\end{block}
\end{blockarray}\qquad\text{B}=\begin{blockarray}{ccc}
1 & 2 \\
\begin{block}{(cc)c}
 1 & 0&1 \\
 1 & 1&2\\
\end{block}
\end{blockarray}\qquad\text{C}=\begin{blockarray}{cccccc}
1&2 & 3 &4&5\\
\begin{block}{(ccccc)c}
 1 & 0 & 0&0&0 &1\\
 1 & 1 & 0&0& 0&2\\
 1 & 1 & 1&0 &0&3\\
 1&1&1&1&0&4\\
 1 & 1 & 1&0&1 &5\\
\end{block}
\end{blockarray}$$

It can be verified that $\text{A}\square_1 \text{B}=\text{A}\square_2\text{B}=\text{C}.$ On the other hand, $\text{A}\square_1 \text{B}\neq\text{A}\square_3 \text{B}\neq\text{A}\square_4 \text{B}$
\begin{remark}{\rm
The subposet matrix $A[\alpha]$ with  $\alpha=\{1,2\}$ forms a totally connected subposet matrix of $A.$ The submatrix  $D=A[{3,4}|{1,2}]$ has $2$ equal columns. By Theorem \ref{connectedsubposet} the output poset matrices from insertion at the labels $\{1,2\}$ are identical. On the other hand, consider the totally connected subposet matrix  $A[\alpha]$ with  $\alpha=\{1,2,3\}.$ In this case its associated submatrix $D=A[\{4\}|\{1,2,3\}]$ do not have all equal columns since the entry of the third column of $D$ is different from the first and second column. By Theorem \ref{connectedsubposet} its output poset matrix from the square partial composition operation at insertion point $3$.would be different from the output poset matrices at insertion points $1$ and $2$ of the input poset matrix $A.$}

\end{remark}

\begin{example}{\rm Consider the poset matrices}
\end{example}
 $$\text{A}=\begin{blockarray}{ccccc}
1 & 2&3&4\\
\begin{block}{(cccc)c}
 1 & 0&0  &0&1\\
 1 & 1&0 &0&2\\
 0 & 0&1 &0&3 \\
 1&1&1&1&4\\
\end{block}
\end{blockarray}\qquad\text{B}=\begin{blockarray}{ccc}
1 & 2 \\
\begin{block}{(cc)c}
 1 & 0&1 \\
 1 & 1&2\\
\end{block}
\end{blockarray}\qquad\text{C}=\begin{blockarray}{cccccc}
1&2 & 3 &4&5\\
\begin{block}{(ccccc)c}
 1 & 0 & 0&0&0 &1\\
 1 & 1 & 0&0& 0&2\\
 1 & 1 & 1&0 &0&3\\
 0&0&0&1&0&4\\
 1 & 1 & 1&1&1 &5\\
\end{block}
\end{blockarray}$$

It can be verified that $\text{A}\square_1 \text{B}=\text{A}\square_2\text{B}=\text{C}.$ On the other hand, $\text{A}\square_1 \text{B}\neq\text{A}\square_3 \text{B}\neq\text{A}\square_4 \text{B}$
\begin{remark}
{\rm The totally connected subposet matrix $A[\alpha]$ with  $\alpha=\{1,2\}$ forms a totally connected subposet matrix of $A.$ The submatrix  $D=A[{3,4}|{1,2}]$ has $2$ equal coluns. By Theorem \ref{connectedsubposet}(i) the outposet matrix from insertion at the labels $\{1,2\}$ are identical. On the other hand,  the totally connected subposet matrix $A[\alpha]$ with  $\alpha=\{3,4\}$ fdoes not result in identical output poset matrices at insertion points $3$ and $4.$  In this case, it can be observed that the associated submatrix $D=A[{3,4}|{1,2}]$ does not have equal rows as required by Theorem \ref{connectedsubposet}(ii).}

\end{remark}

\begin{example}{\rm Consider the poset matrices}
\end{example}
 $$\text{A}=\begin{blockarray}{ccccc}
1 & 2&3&4\\
\begin{block}{(cccc)c}
 1 & 0&0  &0&1\\
 1 & 1&0 &0&2\\
 1 & 0&1 &0&3 \\
 1&0&1&1&4\\
\end{block}
\end{blockarray}\qquad\text{B}=\begin{blockarray}{ccc}
1 & 2 \\
\begin{block}{(cc)c}
 1 & 0&1 \\
 1 & 1&2\\
\end{block}
\end{blockarray}\qquad\text{C}=\begin{blockarray}{cccccc}
1&2 & 3 &4&5\\
\begin{block}{(ccccc)c}
 1 & 0 & 0&0&0 &1\\
 1 & 1 & 0&0& 0&2\\
 1 & 0 & 1&0 &0&3\\
 1&0&1&1&0&4\\
 1 & 0 & 1&1&1 &5\\
\end{block}
\end{blockarray}$$

It can be verified that $\text{A}\square_3 \text{B}=\text{A}\square_4\text{B}=\text{C}.$ On the other hand, $\text{A}\square_3 \text{B}\neq\text{A}\square_1 \text{B}\neq\text{A}\square_2 \text{B}.$
\begin{remark}{\rm The totally connected subposet matrix $A[\alpha]$ with  $\alpha=\{3,4\}$ forms a totally connected subposet matrix of $A.$ The submatrix  $D=A[{3,4}|{1,2}]$ has $2$ equal rows. By Theorem \ref{connectedsubposet}(ii) the outposet matrix from insertion at the labels $\{3,4\}$ are identical. On the other hand,  the totally connected subposet matrix $A[\alpha]$ with  $\alpha=\{1,2\}$ does not result in identical output poset matrices at insertion points $1$ and $2.$  In this case, it can be observed that the associated submatrix $D=A[{3,4}|{1,2}]$ does not have equal columns as required by Theorem \ref{connectedsubposet}(i).
}

\end{remark}

\begin{theorem}\label{disconnectedsubposet} Let $A\in {\cal PM}(n)$ and let $\text{B}$ be a totally disconnected poset matrix of size $m\times m.$  Then the following holds.
\begin{enumerate}
\item  If there exist a totally disconnected subposet matrix $A[\alpha]$ defined on $\alpha=\{1,...,k\}$ with $1<k<n$ such that the submatrix $D:=A[\{k+1,...,n\}|\{1,...,k\}]$ satisfies the condition $D^{(p)}=D^{(p+1)}$ for each $p\in\{1,...,k-1\}.$ Then $\text{A}\square_i \text{B}=\text{A}\square_r \text{B}$ whenever $i,r\in\alpha.$
\item If there exist a totally disconnected subposet matrix $A[\alpha]$ defined on $\alpha=\{k,...,n\}$ with $k\geq 2$ such that the submatrix $D:=A[\{k,...,n\}|\{1,...,k-1\}]$ satisfies the condition $D_{(p)}=D_{(p+1)}$ for each $p\in\{1,...,n-k\}.$ Then $\text{A}\square_i \text{B}=\text{A}\square_r \text{B}$ whenever $i,r\in\alpha.$
\item   If there exist a totally disconnected subposet matrix $A[\alpha]$ defined on $\alpha=\{d,...,k\}$ with $1<d<k<n$ such that the submatrix $D:=A[\{k,...,n\}|\{1,...,k-1\}]$ satisfies the condition $D_{(p)}=D_{(p+1)}$ for each $p\in\{1,...,n-k\}$ where $D_{(p)}$ is a   \textbf{1}-vector or a \textbf{0}-vector of size $1 \times(k-1).$  Then $\text{A}\square_i \text{B}=\text{A}\square_r \text{B}$ whenever $i,r\in\alpha.$
\end{enumerate}
\end{theorem}

\begin{proof}
Similar to Theorem \ref{connectedsubposet}.
\end{proof}

\begin{example}
\end{example}

Consider the poset matrices $$\text{A}=\begin{blockarray}{ccccc}
1 & 2&3&4\\
\begin{block}{(cccc)c}
 1 & 0&0  &0&1\\
 1 & 1&0 &0&2\\
 1 & 1&1 &0&3 \\
 1&1&0&1&4\\
\end{block}
\end{blockarray}\qquad\text{B}=\begin{blockarray}{ccc}
1 & 2 \\
\begin{block}{(cc)c}
 1 & 0&1 \\
 0 & 1&2\\
\end{block}
\end{blockarray}\qquad\text{C}=\begin{blockarray}{cccccc}
1&2 & 3 &4&5\\
\begin{block}{(ccccc)c}
 1 & 0 & 0&0&0 &1\\
 1 & 1 & 0&0& 0&2\\
 1 & 1 & 1&0 &0&3\\
 1&1&0&1&0&4\\
 1 & 1 & 0&0&1 &5\\
\end{block}
\end{blockarray}$$

It can be verified that $\text{A}\square_3 \text{B}=\text{A}\square_4\text{B}=\text{C}.$ On the other hand, $\text{A}\square_1 \text{B}\neq\text{A}\square_2\text{B}\neq\text{A}\square_4 \text{B}.$

\begin{remark}
{\rm The subposet matrix $A[\alpha]$ with  $\alpha=\{3,4\}$ forms a totally disconnected subposet matrix of $A.$ The submatrix  $D=A[{3,4}|{1,2}]$ has $2$ equal rows. By Theorem \ref{disconnectedsubposet}(ii) the output poset matrix derived from the square PCO insertion at the labels $\{3,4\}$ are identical.}
\end{remark}

\begin{example}{\rm 
Consider the poset matrices}
\end{example}
 $$\text{A}=\begin{blockarray}{ccccc}
1 & 2&3&4\\
\begin{block}{(cccc)c}
 1 & 0&0  &0&1\\
 0 & 1&0 &0&2\\
 0 & 0&1 &0&3 \\
 1&1&1&1&4\\
\end{block}
\end{blockarray}\qquad\text{B}=\begin{blockarray}{ccc}
1 & 2 \\
\begin{block}{(cc)c}
 1 & 0&1 \\
 0 & 1&2\\
\end{block}
\end{blockarray}\qquad\text{C}=\begin{blockarray}{cccccc}
1&2 & 3 &4&5\\
\begin{block}{(ccccc)c}
 1 & 0 & 0&0&0 &1\\
 0 & 1 & 0&0& 0&2\\
 0 & 0 & 1&0 &0&3\\
 0&0&0&1&0&4\\
 1 & 1 & 1&1&1 &5\\
\end{block}
\end{blockarray}$$

It can be verified that $\text{A}\square_1 \text{B}=\text{A}\square_2\text{B}=\text{A}\square_3\text{B}=\text{C}.$ On the other hand, $\text{A}\square_1 \text{B}\neq\text{A}\square_4 \text{B}.$

\begin{remark}
{\rm The subposet matrix $A[\alpha]$ with  $\alpha=\{1,2,3\}$ forms a totally disconnected subposet matrix of $A.$ The submatrix  $D=A[{4}|{1,2,3}]$ has $3$ equal columns. By Theorem \ref{disconnectedsubposet}(i) the output poset matrix derived from $\square_i$ insertions at the labels $\{1,2,3\}$ are identical.}

\end{remark}

\section{Structural properties of operads from dual poset matrices}

The {\it dual} ${P^*}$ of a poset $P$ is the poset obtained from $P$ by inverting the order relation on the same elements. If $P=P^*$ then $P$ is called self-dual. Many important posets are self-dual. The dual $P^*$ of a poset $P$ is obtained simply by turning the Hasse diagram of $P$ upside down. Obviously posets are dual in pairs whenever they are not self-dual. Similarly, the definitions and theorems involving posets are dual in pairs, when they are not self-dual; if any theorem is true for all posets, so is its dual.

 Let $A^*$ denote the poset matrix associated to $\mathcal{P^*}$. Clearly, $A^*$ is the poset matrix obtained from the Hasse diagram of $\mathcal{P}$ by replacing label $i$ by $n-i+1$.  Using the $n \times n$ {\it backward identity matrix} defined as
\[
E =
\left[\begin{array}{cccc}
	0 & \cdots & 0 & 1 \\
	\vdots & \udots & 1 & 0 \\
	0 & \udots & \udots & \vdots \\
	1 & 0 & \cdots & 0
\end{array}\right],
\]
the matrix $A^*$ can be obtained from $A^F$ the {\it flip-transpose} of $A$ defined by $A^F=EA^TE$.
\bigskip

For a given $n$ and an index set $\alpha$, let $\alpha^F=\{n-i+1:i \in \alpha\}$. Then the following lemma follows from the definition of $\alpha^F$.
\begin{lemma}\label{dual2}
Let $A \in {\cal PM}(n)$. Then we have:
\begin{itemize}
	\item[(1)] $A[\alpha;\beta]^*=A^*[\beta^F;\alpha^F]$. In particular,  $A[\alpha]^*=A^*[\alpha^F]$.
	\item[(2)] $({\mathbbm{1}}^T_{m}\otimes A_{(i)})^*={\mathbbm{1}}_{m}\otimes A^{*(n-i+1)}$ and $({\mathbbm{1}}_{m}\otimes A^{(i)})^*={\mathbbm{1}}^T_{m}\otimes A^*_{(n-i+1)}$ .
\end{itemize}
\end{lemma}

\begin{theorem}\label{dualthm}
 Let $B\in {\cal PM}(n)$ and $C\in{\cal PM}(m)$. Then we have:
 \begin{itemize}
 \item[(1)] If $A=B\square_i C$ then $A^{*}=B^{*}\square_{n-i+1} C^{*}$.
 \item[(2)] If $A=B \SG{\PosetComposition{0}{1}_i} C$ then $A^{*}=B^{*}
\SG{\PosetComposition{1}{0}_{n - i + 1}} C^{*}$.
 \end{itemize}
 \end{theorem}
  \begin{proof}
 For brevity, let
 \[\alpha_i= \{1,\ldots, i-1\}\quad \text{and} \quad \beta_{n-i}=\{i+1,\ldots, n\}.
 \]
 Then \[\alpha_i^F= \{n-i+2,\ldots, n\}=\beta_{i-1}\quad \text{and} \quad \beta_{n-i}^F=\{1,\ldots, n-i\}=\alpha_{n-i+1}.
 \]
(1) Let $A=B\square_i C$. By lemma \ref{dual2}, we have:
 \begin{align}
 	A^*=(B\square_i C)^* &=
 	\left[\begin{array}{c|c|c}
 		B[\alpha_i] &  \mathbb{O} &  \mathbb{O}\\
 		\hline
 		{\mathbbm{1}}^T_{m}\otimes B_{(i)} & C & \mathbb{O}\\
 		\hline
 		B[\beta_{n-i};\alpha_i] & {\mathbbm{1}}_{m}\otimes B^{(i)} & B[\beta_{n-i}]
 	\end{array}\right]^* \nonumber\\
 	&=
 	\left[\begin{array}{c|c|c}
 		B[\beta_{n-i}]^* & \mathbb{O} & \mathbb{O}\\
 		\hline
 		{\mathbbm{1}}_{m}^T\otimes B^*_{(n-i+1)} & C^* & \mathbb{O}\\
 		\hline
 		B[\beta_{n-i};\alpha_i]^* & {\mathbbm{1}}_{m}\otimes B^{*(n-i+1)} & B[\alpha_i]^*
 	\end{array}\right]\nonumber\\
 	&=
 	\left[\begin{array}{c|c|c}
 		B^*[\beta_{n-i}^F] & \mathbb{O} & \mathbb{O}\\
 		\hline
 		{\mathbbm{1}}_{m}^T\otimes B^*_{(n-i+1)} & C^* & \mathbb{O}\\
 		\hline
 		B^*[\alpha_i^F;\beta_{n-i}^F] & {\mathbbm{1}}_{m}\otimes B^{*(n-i+1)} & B^*[\alpha_i^F]
 	\end{array}\right]\nonumber\\
 &=
 \left[\begin{array}{c|c|c}
 	B^*[\alpha_{n-i+1}] & \mathbb{O} & \mathbb{O}\\
 	\hline
 	{\mathbbm{1}}_{m}^T\otimes B^*_{(n-i+1)} & C^* & \mathbb{O}\\
 	\hline
 	B^*[\beta_{i-1};\alpha_{n-i+1}] & {\mathbbm{1}}_{m}\otimes B^{*(n-i+1)} & B^*[\beta_{i-1}]
 \end{array}\right]\nonumber\\
 	&=B^{*}\square_{n-i+1} C^{*}.\nonumber	
 \end{align}
 (2) It can be also shown in the similar method used in (1).
 	\end{proof}

\begin{corollary} Let $A\in{\cal PM}(n)$ and $B\in{\cal PM}(m)$. Then
$A \square_i B$ is self dual if and only if $A$ and $B$ are self dual.
\end{corollary}

\noindent The following example illustrates the two output poset matrices generated from the partial composition operation $\scalebox{0.7}{$\square$}_i$ when the two input poset matrices of the operations are dual to each other.

\begin{example}
\end{example}

$\begin{blockarray}{cccc}
1 & 2&3\\
\begin{block}{(ccc)c}
 1 & 0&0  &1\\
 1 & 1&0 &2\\
 1 & 0&1 &3 \\
\end{block}
\end{blockarray}\scalebox{0.7}{$\square$}_3\;\;
\begin{blockarray}{ccc}
1 & 2 \\
\begin{block}{(cc)c}
 1 & 0&1 \\
 1 & 1&2\\
\end{block}
\end{blockarray}=
\begin{blockarray}{ccccc}
1&2 & 3 &4\\
\begin{block}{(cccc)c}
 1 & 0 & 0&0 &1\\
 1 & 1 & 0&0 &2\\
 1 & 0 & 1&0 &3\\
 1 & 0 & 1&1 &4\\
\end{block}
\end{blockarray}\quad\Leftrightarrow\quad$
\begin{minipage}[c]{.20\textwidth}
\begin{tikzpicture}
  [scale=.7,auto=center,every node/.style={circle,fill=blue!20}]

\node[rectangle] (c1) at (6.2,2.5) {2};
  \node[rectangle] (c2) at (7,1.5) {1};
  \node[rectangle] (c4) at (7.8,2.5) {3};
  \node[rectangle] (c8) at (7.8,3.7)  {4};

      \draw (c1) -- (c2);
       \draw (c4) -- (c2);
   \draw (c8) -- (c4);

 \end{tikzpicture}

\end{minipage}

The dual construction results to the following.

$\begin{blockarray}{cccc}
1 & 2&3\\
\begin{block}{(ccc)c}
 1 & 0&0  &1\\
 0 & 1&0 &2\\
 1 & 1&1 &3 \\
\end{block}
\end{blockarray}\scalebox{0.7}{$\square$}_1\;\;
\begin{blockarray}{ccc}
1& 2 \\
\begin{block}{(cc)c}
 1 & 0&1 \\
 1 & 1&2\\
\end{block}
\end{blockarray}=
\begin{blockarray}{ccccc}
1&2& 3 &4\\
\begin{block}{(cccc)c}
 1 & 0 & 0&0 &1\\
1 & 1 & 0&0 &2\\
 0 &0 & 1&0 &3\\
 1 & 1 & 1&1 &4\\
\end{block}
\end{blockarray}\quad\Leftrightarrow\quad$
\begin{minipage}[c]{.20\textwidth}
\begin{tikzpicture}
  [scale=.7,auto=center,every node/.style={circle,fill=blue!20}]

\node[rectangle] (c1) at (6.2,3) {3};
  \node[rectangle] (c2) at (7.8,1.7) {1};
  \node[rectangle] (c4) at (7.8,3) {2};
  \node[rectangle] (c8) at (7,4)  {4};

      \draw (c4) -- (c2);
    \draw (c8) -- (c4);
   \draw (c8) -- (c1);

 \end{tikzpicture}

 \end{minipage}

\begin{theorem}
$A\SG{\PosetComposition{1}{0}_{n - i + 1}} B=(A^* \SG{\PosetComposition{0}{1}_i} B^*)^*$ where $A \in \mathcal{PM}(n)$.
\end{theorem}
\begin{proof} Let $\alpha= \{1,\ldots, i-1\}$ and $\beta=\{i+1,\ldots, n\}.$ Then $\alpha^F= \{n-i+2,\ldots, n\}$ and $\beta^F=\{1,\ldots, n-i\}.$ Then it follows from Lemma \ref{dual2} that
\begin{align}
	(A^* \SG{\PosetComposition{0}{1}_i} B^*)^* &=
	\left[\begin{array}{c|c|c}
		A^*[\alpha] & \mathbb{O} & \mathbb{O}\\
		\hline
		\br_{A^*}(i)\otimes {\bf{1}} & B^* & \mathbb{O}\\
		\hline
		A^*[\beta;\alpha] & M & A^*[\beta]
	\end{array}\right]^* \nonumber\\
&=
\left[\begin{array}{c|c|c}
	A[\beta^F] & \mathbb{O} & \mathbb{O}\\
	\hline
	M^* & B & \mathbb{O}\\
	\hline
	A[\alpha^F;\beta^F] & \bc_{A}(n-i+1)\otimes {\bf{1}}^T & A[\alpha^F],
\end{array}\right],\label{eq4}
\end{align}
where $M$ is of the form:
\begin{equation*}
	M[;j] =
	\begin{cases}
		\bf{0} & \text{if $j$ is not minimal associated to $B^*$},\\
		\bc_A(i) & \text{if $j$ is minimal associated to $B^*$}.
	\end{cases}
\end{equation*}
By Lemma \ref{minmax} and Lemma \ref{dual2},
\begin{equation*}
	M^*[j;] =
	\begin{cases}
		\bf{0} & \text{if $j$ is not maximal associated to $B$},\\
		\br_A(n-i+1) & \text{if $j$ is maximal associated to $B$},
	\end{cases}
\end{equation*}
which implies that the right-hand side of equation \eqref{eq4} is equal to $A
\SG{\PosetComposition{1}{0}_{n - i + 1}} B$.
\end{proof}

\begin{definition}
{\rm Let $\text{A}$ and $\text{B}$ be poset matrices of size $n\times n.$  If there exist a subposet matrix $\text{A}^{\prime\prime}$ of A which is a dual disconnected poset matrix of the  subposet matrix $\text{B}^{\prime\prime}$ of $\text{B}$ and the corresponding entries of $A$ and $B$ are always equal except at the region covered by the subposet matrices $\text{A}^{\prime\prime}$  and  $\text{B}^{\prime\prime}$ respectively,  then we shall henceforth refer to  the poset matrix $A$  as a \textbf{semi-equidual} of the poset matrix  $B$ and vice-versa.}
\end{definition}

\begin{example}{\rm Semi-equidual poset matrices of order $4$ and $5$ are as follows:}
\end{example}
$$A=\begin{blockarray}{ccccc}
1&2 & 3 &4\\
\begin{block}{(cccc)c}
 1 & 0 & 0&0 &1\\
 1 & 1 & 0&0 &2\\
 1 & 1 & 1&0 &3\\
 1 & 0 & 0&1 &4\\
\end{block}
\end{blockarray}\quad B=\begin{blockarray}{ccccc}
1&2 & 3 &4\\
\begin{block}{(cccc)c}
 1 & 0 & 0&0 &1\\
 1 & 1 & 0&0 &2\\
 1 & 0 & 1&0 &3\\
 1 & 0 & 1&1 &4\\
\end{block}
\end{blockarray}.$$

$$C=\begin{blockarray}{cccccc}
1&2 & 3 &4&5\\
\begin{block}{(ccccc)c}
 1 & 0 & 0&0&0 &1\\
 0 & 1 & 0&0& 0&2\\
 0 & 1 & 1&0 &0&3\\
 1&1&1&1&0&4\\
 1 & 1 & 1&1&1 &5\\
\end{block}
\end{blockarray}\quad D=\begin{blockarray}{cccccc}
1&2 & 3 &4&5\\
\begin{block}{(ccccc)c}
 1 & 0 & 0&0&0 &1\\
 1 & 1 & 0&0& 0&2\\
 0 & 0 & 1&0 &0&3\\
 1&1&1&1&0&4\\
 1 & 1 & 1&1&1 &5\\
\end{block}
\end{blockarray}.$$

\begin{remark}
{\rm The poset matrices $A$ and $B$ are semi-equidual poset matrices since as the disconnected subposet matrix $A[\{2,3,4\}]$  is a dual of the disconnected subposet matrix $B[\{2,3,4\}]$  and all entries of the column $1$  and row $1$ of $A$ are equal to the corresponding entries of $B$. Similarly, the poset matrices $C$ and $D$ are semi-equidual since as the disconnected subposet matrix $C[\{1,2,3\}]$ is a dual to the disconnected subposet matrix $D[\{1,2,3\}]$ and all the entries in the fourth row, fourth column, fifth row and fifth column of $C$ are equal to the corresponding entries of $D.$}
\end{remark}

\begin{theorem}\label{semiequidual}
 Let $A\in {\cal PM}(n)$ and let $\text{B}$ be a totally connected poset matrix of size $m\times  m.$Then the following holds.
\begin{enumerate}
\item If there exist a totally disconnected subposet matrix $A[\alpha]$ defined on $\alpha=\{1,...,k\}$ with $1<k<n$ such that the submatrix $D:=A[\{k+1,...,n\}|\{1,...,k\}]$ satisfies the condition that $D^{(p)}=D^{(p+1)}$ for each $p\in\{1,...,k-1\}.$ Then $\text{A}\square_1\text{B}$ \text{ is a semi-equidual poset matrix of} $\text{A}\square_k \text{B}$ and vice-versa.
\item If there exist a totally disconnected subposet matrix $A[\alpha]$ defined on $\alpha=\{k,...,n\}$ with $k\geq 2$ such that the submatrix $D:=A[\{k,...,n\}|\{1,...,k-1\}]$ satisfies the condition that $D_{(p)}=D_{(p+1)}$ for each $p\in\{1,...,n-k\}.$ Then $\text{A}\square_k\text{B}$ \text{ is a semi-equidual poset matrix of} $\text{A}\square_n \text{B}$ and vice-versa.
	\item If there exist a totally disconnected subposet matrix $A[\alpha]$ defined on $\alpha=\{d,...,k\}$ with $1<d<k<n$ such that the submatrix $D:=A[\{k,...,n\}|\{1,...,k-1\}]$ satisfies the condition that $D_{(p)}=D_{(p+1)}$ for each $p\in\{1,...,n-k\}$ where $D_{(p)}$ is a   \textbf{1}-vector or a \textbf{0}-vector of size $1 \times(k-1).$ Then $\text{A}\square_d\text{B}$ \text{ is a semi-equidual poset matrix of} $\text{A}\square_k \text{B}$ and vice-versa.
\end{enumerate}
\end{theorem}

\begin{proof}
Let $1,k\in\alpha$  as stated  in (1).  Consider that the output poset matrix structure of $A\square_{1}B=E$ can be partitioned into three submatrices $E1, E2$ and $E3$ such that:
\begin{itemize}
\item $E1=E[\{1,...,k+m-1\}]$  of size $(k+m-1)\times (k+m-1).$
\item $E2=E[\{k+m,...,n+m-1\}\mid\{1,...,k+m-1\}]$ of size $(n-k)\times (k+m-1).$
\item $E3=E[\{k+m,...,n+m-1\}]$ of size $(n-k)\times(n-k).$
\end{itemize}
\noindent $E1$ can be subpartitioned such that $E[\{1,...,m\}]=B$, $E[\{1+m,...,k+m-1\}\mid\{1,...,m\} ]$ is a zero matrix of size $(k-1)\times m,$ and $E[\{m+1,...,k+m-1\}]$ is either the connected poset matrix of size $1$ when $k=2$ or the totally disconnected poset matrix of size $k-1$ when $k>2.$

\noindent$E2$ can  be subpartitioned such that $E[\{k+m,...,n+m-1\}\mid\{1,...,m\}]={\mathbbm{1}}_{m}\otimes A^{(1)}[\{k+1,...,n\}\mid\{1,..,k\}]$,  and $E[\{k+m,...,n+m-1\}\mid\{m+1,...,k+m-1\}]=A[\{k+1,...,n\}\mid\{2,...,k\}].$

\noindent$E3=A[\{k+1,...,n\}].$

At the insertion point $k$, similarly consider that the output poset matrix structure of $A\square_{k}B=F$ can be partitioned into three submatrices $F1, F2$ and $F3$ such that:
\begin{itemize}
\item $F1=F[\{1,...,k+m-1\}]$  of size $(k+m-1)\times (k+m-1).$
\item $F2=F[\{k+m,...,n+m-1\}\mid\{1,...,k+m-1\}]$ of size $(n-k)\times (k+m-1).$
\item $F3=F[\{k+m,...,n+m-1\}]$ of size $(n-k)\times(n-k).$
\end{itemize}
\noindent $F1$ can be subpartitioned such that $F[\{1,...,k-1\}]$  is either the connected poset matrix of size $1$ when $k=2$ or the totally disconnected poset matrix of size $k-1$ when $k>2$
, $F[\{k,...,k+m-1\}\mid\{1,...,k-1\} ]$ is a zero matrix of size $m\times (k-1),$ and $F[\{k,...,k+m-1\}]=B.$

\noindent$F2$ can  be subpartitioned such that $F[\{k+m,...,n+m-1\}\mid\{k,...,k+m\}]={\mathbbm{1}}_{m}\otimes A^{(k)}[\{k+1,...,n\}\mid\{1,..,k\}]$, $F[\{k+m,...,n+m-1\}\mid\{1,...,k-1\}]=A[\{k+1,...,n\}\mid\{1,..,k-1\}].$

\noindent$F3=A[\{k+1,...,n\}].$

\noindent From the matrix structure of $F$ and $E,$ it follows that:
\begin{itemize}
\item For $E1=(e_{i,j})$ and $F1=(f_{i,j})$ both of size $k+m-1,$  it can be verified that $F1$ is a dual poset matrix of $E1$ since $F1=(e_{{k+m-j},{k+m-i}})$ and $E1=(f_{{k+m-j},{k+m-i}}).$ By definition \ref{disconnectMAT},  $F1$ and $E1$ are both disconnected poset matrix structure.
\item  $E2=F2$ since as each of the columns in the submatrix $A[\{k+1,...,n\}|\{1,...,k\}]$ has equal corresponding entries.
\item  $E3=F3.$
\end{itemize}

Thus, $E$ is a semi-equidual poset matrix of $F$ and vice-versa.

\noindent Let $k,n\in\alpha$  as stated  in (2).  Consider that the output poset matrix structure of $A\square_{k}B=G$ can be partitioned into three submatrices $G1, G2$ and $G3$ such that:
\begin{itemize}
\item $G1=G[\{1,...,k-1\}]$  of size $(k-1)\times (k-1).$
\item $G2=G[\{k,...,n+m-1\}\mid\{1,...,k-1\}]$ of size $(n+m-k)\times (k-1).$
\item $G3=G[\{k,...,n+m-1\}]$ of size $(n+m-k)\times(n+m-k).$
\end{itemize}
\noindent $G3$ can be subpartitioned such that $G[\{k,...,k+m\}]=B$, $G[\{k+m-1,...,n+m-1\}\mid\{k,...,k+m-1\} ]$  is a zero matrix of size $(n-k)\times m,$ and $G[\{k+m-1,...,n+m-1\}]$ is either the connected poset matrix of size $1$ when $k=2$ or the totally disconnected poset matrix of size $(n-k)\times(n-k)$ when $k>2.$

\noindent$G2$ can  be subpartitioned such that $G[\{k,...,k+m-1\}\mid\{1,...,k-1\}]={\mathbbm{1}}_m^T\otimes A_{(k)}[\{k,...,n\}\mid\{1,..,k-1\}]$,  and $G[\{k+m,...,n+m-1\}\mid\{1,...,k-1\}]=A[\{k,...,n\}\mid\{1,...,k-1\}].$

\noindent$G1=A[\{1,...,k-1\}].$

 At the insertion point $n,$ consider that the output poset matrix structure of $A\square_{n}B=H$ can be partitioned into three submatrices $H1, H2$ and $H3$ such that:
\begin{itemize}
\item $H1=H[\{1,...,k-1\}]$  of size $(k-1)\times (k-1).$
\item $H2=H[\{k,...,n+m-1\}\mid\{1,...,k-1\}]$ of size $(n+m-k)\times (k-1).$
\item $H3=H[\{k,...,n+m-1\}]$ of size $(n+m-k)\times(n+m-k).$
\end{itemize}
\noindent $H3$ can be subpartitioned such that $H[\{n,...,n+m-1\}]=B$, $H[\{n,...,n+m-1\}\mid\{k,...,n-1\} ]$ is the zero matrix of size $m\times (n-k),$ and $H[\{k,...,n-1\}]$ is either the connected poset matrix of size $1$ when $k=2$ or the totally disconnected poset matrix of size $(n-k)\times(n-k)$ when $k>2.$

\noindent$H2$ can  be subpartitioned such that $H[\{n,...,n+m-1\}\mid\{1,...,k-1\}]={\mathbbm{1}}_m^T\otimes A_{(n)}[\{k,...,n\}\mid\{1,..,k-1\}]$,  and $H[\{k,...,n-1\}\mid\{1,...,k-1\}]=A[\{k,...,n-1\}\mid\{1,...,k-1\}].$

\noindent$H1=A[\{1,...,k-1\}].$

\noindent From the matrix structure of $G$ and $H,$ it follows that:
\begin{itemize}
\item For $G3=(g_{i,j})$ and $H3=(h_{i,j})$ both of size $(n+m-k)\times(n+m-k),$  it can be verified that $H3$ is a dual poset matrix of $G3$ since $H3=(g_{{n+m-k+1-j},{n+m-k+1-i}})$ and $G3=(h_{{n+m-k+1-j},{n+m-k+1-i}}).$ By definition \ref{disconnectMAT},  $H3$ and $G3$ are disconnected poset matrix structures.
\item  $H2=G2.$ under the condition that each of the rows in the submatrix $A[\{k,...,n\}|\{1,...,k-1\}]$ has equal corresponding entries.
\item  $H3=G3.$
\end{itemize}
Thus,  $G$ is a semi-equidual poset matrix of $H$ and vice-versa.

The proof of (3) follows from the arguments in (1) and (2).

\end{proof}

\begin{example}
\end{example}
$\begin{blockarray}{ccccc}
1 & 2&3&4\\
\begin{block}{(cccc)c}
 1 & 0&0  &0&1\\
 1 & 1&0 &0&2\\
 1 & 0&1 &0&3 \\
 1&0&0&1&4\\
\end{block}
\end{blockarray}\scalebox{0.7}{$\square$}_2\;\;
\begin{blockarray}{ccc}
1 & 2 \\
\begin{block}{(cc)c}
 1 & 0&1 \\
 1 & 1&2\\
\end{block}
\end{blockarray}=
\begin{blockarray}{cccccc}
1&2 & 3 &4&5\\
\begin{block}{(ccccc)c}
 1 & 0 & 0&0&0 &1\\
 1 & 1 & 0&0& 0&2\\
 1 & 1 & 1&0 &0&3\\
 1&0&0&1&0&4\\
 1 & 0 & 0&0&1 &5\\
\end{block}
\end{blockarray}\quad\Leftrightarrow\quad$
\begin{minipage}[c]{.20\textwidth}
\begin{tikzpicture}
  [scale=.7,auto=center,every node/.style={circle,fill=blue!20}]

\node[rectangle] (c1) at (5.5,3) {2};
\node[rectangle] (c7) at (5.5,4.5) {3};
  \node[rectangle] (c2) at (7,1.5) {1};
  \node[rectangle] (c4) at (7,3) {4};
   \node[rectangle] (c5) at (8.5,3) {5};
      \draw (c5) -- (c2);
      \draw (c1) -- (c2);
        \draw (c2) -- (c4);
         \draw (c1) -- (c7);

 \end{tikzpicture}

\end{minipage}

$\begin{blockarray}{ccccc}
1 & 2&3&4\\
\begin{block}{(cccc)c}
 1 & 0&0  &0&1\\
 1 & 1&0 &0&2\\
 1 & 0&1 &0&3 \\
 1&0&0&1&4\\
\end{block}
\end{blockarray}\scalebox{0.7}{$\square$}_4\;\;
\begin{blockarray}{ccc}
1 & 2 \\
\begin{block}{(cc)c}
 1 & 0&1 \\
 1 & 1&2\\
\end{block}
\end{blockarray}=
\begin{blockarray}{cccccc}
1&2 & 3 &4&5\\
\begin{block}{(ccccc)c}
 1 & 0 & 0&0&0 &1\\
 1 & 1 & 0&0& 0&2\\
 1 & 0 & 1&0 &0&3\\
 1&0&0&1&0&4\\
 1 & 0 & 0&1&1 &5\\
\end{block}
\end{blockarray}\quad\Leftrightarrow\quad$
\begin{minipage}[c]{.20\textwidth}
\begin{tikzpicture}
  [scale=.7,auto=center,every node/.style={circle,fill=blue!20}]

\node[rectangle] (c1) at (5.5,3) {2};
\node[rectangle] (c7) at (8.5,4.5) {5};
  \node[rectangle] (c2) at (7,1.5) {1};
  \node[rectangle] (c4) at (7,3) {3};
   \node[rectangle] (c5) at (8.5,3) {4};
      \draw (c5) -- (c2);
      \draw (c1) -- (c2);
        \draw (c2) -- (c4);
         \draw (c5) -- (c7);

 \end{tikzpicture}

\end{minipage}

 \begin{remark}
{\rm By applying Theorem  \ref{semiequidual}(ii), we can observe that  the input disconnected subposet matrix $A[ \{2,3,4 \}]$ forms semi-equidual poset matrices at insertion points $2$ and $4,$ and we can observe that the submatrix  $D=A[ \{2,3,4 \}| \{1 \}]$ has all equal rows.}
 \end{remark}

\begin{example}
\end{example}
$\begin{blockarray}{ccccc}
1 & 2&3&4\\
\begin{block}{(cccc)c}
 1 & 0&0  &0&1\\
 0 & 1&0 &0&2\\
 1 & 1&1 &0&3 \\
 1&1&1&1&4\\
\end{block}
\end{blockarray}\scalebox{0.7}{$\square$}_1\;\;
\begin{blockarray}{ccc}
1 & 2 \\
\begin{block}{(cc)c}
 1 & 0&1 \\
 1 & 1&2\\
\end{block}
\end{blockarray}=
\begin{blockarray}{cccccc}
1&2 & 3 &4&5\\
\begin{block}{(ccccc)c}
 1 & 0 & 0&0&0 &1\\
 1 & 1 & 0&0& 0&2\\
 0 & 0 & 1&0 &0&3\\
 1&1&1&1&0&4\\
 1 & 1 & 1&1&1 &5\\
\end{block}
\end{blockarray}\quad\Leftrightarrow\quad$
\begin{minipage}[c]{.20\textwidth}
\begin{tikzpicture}
  [scale=.7,auto=center,every node/.style={circle,fill=blue!20}]

\node(d4) at (10,2) {5};
\node(d1) at (10,0.5) {4};
\node(d6) at (11,-0.3)  {2};
\node(d9) at (11,-1.5)  {1};
\node(d7) at (9,-0.3) {3};
\draw (d6) -- (d1);
\draw (d1) -- (d7);
\draw (d1) -- (d4);
\draw (d9) -- (d6);

 \end{tikzpicture}

\end{minipage}

$\begin{blockarray}{ccccc}
1 & 2&3&4\\
\begin{block}{(cccc)c}
 1 & 0&0  &0&1\\
 0 & 1&0 &0&2\\
 1 & 1&1 &0&3 \\
 1&1&1&1&4\\
\end{block}
\end{blockarray}\scalebox{0.7}{$\square$}_2\;\;
\begin{blockarray}{ccc}
1 & 2 \\
\begin{block}{(cc)c}
 1 & 0&1 \\
 1 & 1&2\\
\end{block}
\end{blockarray}=
\begin{blockarray}{cccccc}
1&2 & 3 &4&5\\
\begin{block}{(ccccc)c}
 1 & 0 & 0&0&0 &1\\
 0 & 1 & 0&0& 0&2\\
 0 & 1 & 1&0 &0&3\\
 1&1&1&1&0&4\\
 1 & 1 & 1&1&1 &5\\
\end{block}
\end{blockarray}\quad\Leftrightarrow\quad$
\begin{minipage}[c]{.20\textwidth}
\begin{tikzpicture}
  [scale=.7,auto=center,every node/.style={circle,fill=blue!20}]

\node(d4) at (10,2) {5};
\node(d1) at (10,0.5) {4};
\node(d6) at (11,-0.3)  {1};
\node(d7) at (9,-0.3) {3};
\node(d9) at (9,-1.5) {2};
\draw (d6) -- (d1);
\draw (d1) -- (d7);
\draw (d1) -- (d4);
\draw (d9) -- (d7);

 \end{tikzpicture}

\end{minipage}

 \begin{remark}
{\rm By applying Theorem  \ref{semiequidual}(i), we can observe that  the input disconnected subposet matrix $A[ \{1,2 \}]$ forms semi-equidual poset matrices at insertion points $1$ and $2,$ and we can observe that the submatrix  $D=A[ \{3,4 \}| \{1,2 \}]$ is comprised of equal columns.}
 \end{remark}

\end{document}